\documentclass[11pt]{article}
\usepackage[T1]{fontenc}
\usepackage{lmodern}
\usepackage{microtype}
\usepackage[reqno]{amsmath}
\usepackage{amssymb,amsthm,mathtools,mathrsfs}
\mathtoolsset{showonlyrefs}
\usepackage[margin=1in]{geometry}
\usepackage[dvipsnames]{xcolor}
\usepackage[colorlinks=true,linkcolor=blue,citecolor=blue,urlcolor=blue]{hyperref}
\usepackage{bookmark}
\usepackage{enumitem,needspace}
\setlist[enumerate]{label=(\roman*),leftmargin=2em}
\newtheorem{theorem}{Theorem}[section]
\newtheorem{lemma}[theorem]{Lemma}
\newtheorem{proposition}[theorem]{Proposition}
\newtheorem{corollary}[theorem]{Corollary}
\theoremstyle{remark}
\newtheorem{remark}[theorem]{Remark}
\newcommand{\Z}{\mathbb Z}
\newcommand{\R}{\mathbb R}
\newcommand{\C}{\mathbb C}
\newcommand{\T}{\mathbb T}
\newcommand{\N}{\mathbb N}
\newcommand{\E}{\mathbb E}
\newcommand{\PP}{\mathbb P}
\newcommand{\ind}{\mathbf 1}
\newcommand{\cS}{\mathcal S}
\newcommand{\cH}{\mathcal H}
\newcommand{\cO}{\mathcal O}
\newcommand{\Bnorm}[1]{\left\|#1\right\|_{B}}
\newcommand{\knorm}[1]{\left\|#1\right\|_K}
\DeclareMathOperator{\PV}{PV}
\DeclareMathOperator{\Law}{Law}
\DeclareMathOperator{\Res}{Res}
\DeclareMathOperator{\lcm}{lcm}
\DeclareMathOperator{\Var}{Var}
\newcommand{\cP}{\mathcal P}
\newcommand{\cF}{\mathcal F}
\newcommand{\cR}{\mathcal R}
\newcommand{\Q}{\mathbb Q}
\newcommand{\cD}{\mathscr D}
\numberwithin{equation}{section}
\allowdisplaybreaks[1]
\title{Spacing statistics for a point scatterer on the cubic three-torus}
\author{Christopher Lutsko}
\date{}
\hypersetup{
  pdftitle={Spacing statistics for a point scatterer on the cubic three-torus},
  pdfauthor={Christopher Lutsko},
  pdfkeywords={point scatterers, Seba billiards, spectral statistics, level repulsion,
               sums of three squares, MSC 2020: 81Q50, 11E25, 35P20}
}
\begin{document}
\makeatletter
\def\Hy@writebookmark#1#2#3#4#5{}
\makeatother

\maketitle
\vspace{4mm}

\begin{abstract}
We study the new eigenvalues of a fixed point interaction on the cubic
three-torus. Their mean-one consecutive spacings converge to a probability
law independent of the interaction parameter. We characterize this law
by the consecutive zeros of a random meromorphic function built from
three-squares congruence densities. Its small-gap distribution has the form
$s^5\Psi(\log_2(s^{-2}))+o(s^5)$, where $\Psi$ is continuous, positive,
and one-periodic. The proof combines periodic mean-square approximation,
a weighted Hilbert-transform estimate, and Fourier estimates that retain
the interactions between primes.
\end{abstract}

\begingroup
\begin{NoHyper}
\renewcommand{\thefootnote}{}
\end{NoHyper}
\endgroup

\section{Introduction}

A central question in quantum chaos is: how does the spectral information
of a quantum system reflect the geometric and dynamical properties of the
corresponding classical system? For instance, the Berry--Tabor
\cite{BT77} and Bohigas--Giannoni--Schmit \cite{BGS} conjectures predict
that generic integrable systems have Poisson local statistics of their
quantum energy levels, whereas generic chaotic systems have random
matrix statistics, with the ensemble determined by the symmetries.
These conjectures motivated the study of explicit models where something
can be said about the quantum energy levels. One notable example is the
\v{S}eba billiard \cite{S90}, and its toral analogue, a point scatterer
on a flat torus. The spectrum retains the Laplace \emph{old eigenvalues}, with each multiplicity reduced by one, and
contains a sequence of \emph{new eigenvalues}. The local statistics of
these new eigenvalues have been closely studied \cite{RU,Kurlberg}.
In this article, we give a limiting gap distribution for the new spectrum
and prove a sharp asymptotic for the frequency of small gaps.

Let $\T^3:=\R^3/(2\pi\Z)^3$ denote the three-torus. The positive
Laplace eigenvalues are the integers $m>0$ such that $r_3(m)>0$, where
\[
 r_3(m):=\#\{v\in\Z^3:|v|^2=m\}.
\]
The multiplicity at $m$ is $r_3(m)$. Fix a point $o\in\T^3$.
The Hamiltonian with a point scatterer is formally written as
\[
 H_{\alpha,o}:=-\Delta+\alpha\delta_o.
\]
It is defined rigorously through a nontrivial self-adjoint extension of
the Laplacian restricted away from $o$. We use $c\in\R$ for the fixed
extension parameter in the spectral equation \eqref{conv:eq:spectral}.
The formal coupling $\alpha$ will not be used below.

Between every two consecutive distinct old eigenvalues there is
exactly one new eigenvalue. Thus the new eigenvalues \emph{interlace}
the old ones. Write
\[
 \Lambda_1(c)<\Lambda_2(c)<\cdots
\]
for the positive new eigenvalues in increasing order. The new spectrum
has density $\rho:=5/6$, as recalled in Section~\ref{sec:normalization}.
Its mean-one spacings and empirical spacing measure are
\[
 s_j(c):=\frac56\bigl(\Lambda_{j+1}(c)-\Lambda_j(c)\bigr),
 \qquad \mu_N^{(c)}:=\frac1N\sum_{j=1}^N\delta_{s_j(c)}.
\]

\Needspace{7\baselineskip}
\begin{theorem}\label{conv:thm:main}
There is a probability measure $\mu$ on $[0,10/3]$ such that, for every
fixed $c\in\R$,
\begin{equation}\label{conv:eq:limit}
 \mu_N^{(c)}\ \Longrightarrow\ \mu\qquad(N\to\infty).
\end{equation}
The limit is independent of $c$, has mean one, and satisfies
$\mu(\{0\})=0$.
\end{theorem}
The convergence in \eqref{conv:eq:limit} requires no averaging over the
interaction parameter. In particular, at every continuity point $s$ of
$F(s):=\mu([0,s])$,
\[
 \frac1N\#\{j\leq N:s_j(c)\leq s\}\longrightarrow F(s).
\]
An explicit description of $\mu$ is given in \eqref{conv:eq:law}.
The bounded support comes from interlacing.

Our second result concerns the mass that this law assigns to small gaps.
\Needspace{6\baselineskip}
\begin{theorem}\label{thm:spacing}
There is a continuous, strictly positive, one-periodic function $\Psi$
such that
\begin{equation}\label{eq:unfoldedasymptotic}
 F(s)=s^5\Psi\bigl(\log_2(s^{-2})\bigr)+o(s^5)
 \qquad(s\downarrow0).
\end{equation}
In particular, $F(s)\asymp s^5$.
\end{theorem}

This gives a strong form of level repulsion: the limiting law assigns
mass of order $s^5$ to $[0,s]$. For comparison, the cumulative
Poisson spacing law is of order $s$, and the GOE law is of order $s^2$
near zero. The periodic coefficient reflects the powers of four in the
three-squares problem. 

\subsection{Relation to earlier work}
The model originates with \v{S}eba \cite{S90}, who studied a rectangular
billiard with a point interaction and observed spectral and eigenfunction
behavior associated with quantum chaos, despite the essentially
integrable classical motion. Shigehara \cite{SH94} clarified the role
of the coupling: in two dimensions, the regime displaying repulsion
requires the extension parameter to vary with energy. A fixed extension
has a different high-energy limit. The toral model replaces the billiard
boundary by periodic boundary conditions.

Rudnick--Uebersch\"ar \cite{RU} made this distinction precise using a
trace formula. In two dimensions and at fixed coupling, the new levels
approach the distinct old levels on the mean-spacing scale. Consequently,
their limiting spacing distributions agree whenever either exists.
In three dimensions, they proved instead that the mean old-to-new
displacement is asymptotic to half the mean old spacing. This provides
information about the average displacement without determining the
spacing distribution.

For the cubic three-torus, Kurlberg \cite{Kurlberg} proved that the
upper asymptotic proportion of gaps at most $s$ is
$O_\varepsilon(s^{4-\varepsilon})$ for every $\varepsilon>0$.
This bound does not assume existence of a limiting spacing law.
His argument relates small gaps to unusually small normalized
multiplicities and identifies divisibility by large powers of four as
the principal arithmetic mechanism. 
Kurlberg also discusses a sixth-order heuristic; the fifth-order law
here comes from analyzing the regular Green value together with the
central multiplicity and the slope.

Alongside this work on the distribution of the eigenvalues, the distribution
and shape of the new eigenfunctions have also been studied.  Kurlberg--Uebersch\"ar
proved quantum ergodicity on arithmetic tori and, in contrast,
established superscars on Diophantine irrational tori \cite{KUQE,KUS}.

Our proof uses comparison with an arithmetic probability model which is in the
spirit of Kurlberg--Lester \cite{KL}, who proved Poisson angular spacings
for lattice points on circles along a density-one sequence of admissible
radii. Their observable and limiting law are different. In the present
problem, the congruence model preserves the dependence between neighboring
multiplicities.

\subsection{Proof outline}
There are two separate parts. For convergence, put
\[
 a(n):=\frac{r_3(n)}{2\pi\sqrt n},\qquad
 P_M(u):=M^{-2}\#\{x\bmod M:|x|^2\equiv u\pmod M\}.
\]
The second-moment formula of Choi--Kumchev--Osburn \cite{CKO}
gives, with $M_Q:=\lcm(1,\ldots,Q)$,
\[
 \lim_{X\to\infty}\frac1X\sum_{n\leq X}|a(n)-P_{M_Q}(n)|^2
 \ll Q^{-1}.
\]
The new eigenvalues solve $G(\lambda)=c$, where
\[
 G(\lambda):=\sum_{n\geq0}r_3(n)
       \left(\frac1{n-\lambda}-\frac n{n^2+1}\right)
\]
is the regularized scalar Green function. A weighted discrete Hilbert
transform transfers the approximation of $a$ to this spectral equation:
\[
 \frac{G(m+z)-c}{2\pi\sqrt m}
 \approx-\frac{\pi}{M_Q}\sum_{b\bmod M_Q}P_{M_Q}(m+b)
                    \cot\frac{\pi(z-b)}{M_Q}.
\]
We take the limit on the profinite integers. Its surviving poles agree
almost surely with the three-squares criterion; consequently its two
zeros adjacent to the central pole determine the limiting spacing law.

We next explain the small-gap argument.  Fix a realization $H$ of the
limiting Green function for which the origin is a pole.  Let $-u<0<v$ be
the two zeros adjacent to this pole, and put $d:=u+v$.  Thus $d$ is the
gap whose distribution we study.  Near the origin we write
\[
 H(z)=-\frac{R}{z}+h+Dz+O(Dz^2),
\]
where $R>0$ is the magnitude of the residue, $h$ is the value at the
origin of the regular part of $H$, and $D>0$ is the derivative of
the regular part.  Solving the corresponding quadratic equation gives
\[
 d=q+O(q^2),\qquad q:=\frac{\sqrt{h^2+4DR}}{D}.
\]
Consequently, a gap of size at most $s$ requires both
\[
 R\ll Ds^2,\qquad |h|\ll Ds.
\]

The residue inequality $R\ll Ds^2$ is controlled by the $2$-adic
coordinate.  Conditional on the presence of the central pole, let
$\ell\geq0$ be the integer for which this coordinate is divisible by
$4^\ell$ but not by $4^{\ell+1}$.  Then
\[
 R=2^{-\ell}R_0,
 \qquad \PP(\ell=n)=\frac34\,4^{-n},
\]
where $R_0$ is a positive random factor whose law does not depend on
$\ell$.  The scale $R\asymp s^2$ therefore has probability of order
$s^4$.  We then prove that the regular value $h$, with the necessary
weights in $D$ and $R_0$, has a continuous density which is positive at
zero.  Hence the restriction $|h|\ll Ds$ contributes one further factor
of $s$.  This gives the fifth-order law; the discrete choices of $\ell$
produce the periodic factor in \eqref{eq:unfoldedasymptotic}.

\subsection*{Notation}\hypertarget{outline.notation}{}

Write $e(x):=e^{2\pi ix}$, and let $\ind_E$ denote the indicator of $E$.
For a prime $p$, $v_p(u)$ is the $p$-adic valuation, with $v_p(0):=\infty$.
The symbols $f\ll g$ and $f=O(g)$ mean that $|f|$ is bounded by a
constant multiple of $g$; subscripts indicate permitted dependence.
We write $f\asymp g$ when both inequalities hold.

For $\cD:=\C\setminus\Z$, $\cO(\cD)$ denotes the holomorphic functions on
$\cD$, with locally uniform convergence. If $K\subset\cD$ is compact,
$\|f\|_K:=\sup_{z\in K}|f(z)|$. Principal values are symmetric:
\[
 \PV\sum_{k\in\Z}b_k:=\lim_{L\to\infty}\sum_{|k|\leq L}b_k.
\]

\section{Preliminaries}\label{sec:preliminaries}
\subsection{The spectral equation}\label{sec:spectral}
Translation of the torus moves the scatterer to the origin without
changing the spectrum, henceforth we take $o=0$. The eigenfunctions of $-\Delta$
are $e^{iv\cdot x}$ (we refer to the spectrum of $-\Delta$ without any singularity as the \emph{free spectrum}), with eigenvalues $|v|^2$. On each free eigenspace,
evaluation at the origin has a kernel of codimension one. This kernel
is the surviving old eigenspace. The new eigenfunctions are multiples
of the Green function
\[
 g_w(x):=\frac1{(2\pi)^3}\sum_{v\in\Z^3}
                  \frac{e^{iv\cdot x}}{|v|^2-w},
\]
with convergence in $L^2$ when $w$ is off the free spectrum.
The regularized value at the scatterer gives the scalar spectral equation;
see \cite{RU,Kurlberg} for the self-adjoint extension construction.

Put
\[
 r(n):=r_3(n),\qquad
 \cS:=\{n\geq0:r(n)>0\},\qquad \rho:=\frac56.
\]
In particular $r(0)=1$. For a fixed finite real parameter $c$, consider
\begin{equation}\label{conv:eq:spectral}
 G(w)=c, \qquad \text{ where }\qquad
 G(w):=\sum_{n\geq0}r(n)
 \left(\frac1{n-w}-\frac{n}{n^2+1}\right) .
\end{equation}
This is the spectral equation for a fixed nontrivial self-adjoint
extension, in the normalization used here; see \cite{Kurlberg}.
For a compact set of $w$ away from the poles and large $n$,
\[
 \left|\frac1{n-w}-\frac n{n^2+1}\right|\ll_Kn^{-2},\qquad
 \sum_{n\asymp Y}r(n)n^{-2}\ll Y^{-1/2}.
\]
The second estimate follows from the elementary lattice bound
$\sum_{n\leq Y}r(n)\ll Y^{3/2}$. The dyadic tail is summable,
so the series converges locally normally. Differentiation gives
\[
 G'(x)=\sum_{n\geq0}\frac{r(n)}{(n-x)^2}>0.
\]
At consecutive poles, $G$ runs from $-\infty$ to $+\infty$, which proves
interlacing and simplicity of the new eigenvalues. For $m\in\cS$, $m>0$,
write $m_-$ and $m_+$ for the previous and next elements of $\cS$.
There is exactly one solution $\lambda_m(c)\in(m_-,m)$.
This pole labeling differs from the ordinal labeling $\Lambda_j(c)$ in
the introduction: if $m_j$ is the $j$th positive represented integer,
then $\lambda_{m_j}(c)=\Lambda_j(c)$.
Set
\[
 s_m(c):=\rho\bigl(\lambda_{m_+}(c)-\lambda_m(c)\bigr),\qquad
 N(X):=\#(\cS\cap(0,X]).
\]
We study only the new eigenvalues, counted once each. In the proof,
$s_m$ is indexed by the pole $m$; thus $s_{m_j}$ is the $j$th gap.

\subsection{Arithmetic normalization}\label{sec:normalization}
The three-squares criterion and the disjointness of its exceptional
progressions give
\begin{align}
 \N\setminus\cS&=\bigsqcup_{a\geq0}\{4^a(8b+7):b\geq0\},\notag\\
 X-N(X)&=\frac X8\sum_{a\geq0}4^{-a}+O(\log(2X))
        =\frac X6+O(\log(2X)).\label{eq:countS}
\end{align}
Here the finitely many progressions meeting $[1,X]$ each contribute
an $O(1)$ error; the remaining geometric tail contributes $O(1)$.
In particular, $N(X)=\rho X+O(\log(2X))$.

The three-squares criterion gives old gaps between $1$ and $3$.
Moreover, $m-1$ and $m+1$ cannot both be omitted, so $m_+-m_-\leq4$.
Interlacing therefore gives
\begin{equation}\label{conv:eq:bounded}
  0<s_m(c)<4\rho=10/3.
\end{equation}
Finally, three squares summing to a multiple of four must all be even.
Dividing their coordinates by two gives
\begin{equation}\label{eq:fouridentity}
 r(4n)=r(n),\qquad
 \frac{r(4n)}{2\pi\sqrt{4n}}=\frac12\frac{r(n)}{2\pi\sqrt n}.
\end{equation}

\subsection{Congruence probability}
The profinite integers form the set of infinite sequences satisfying
the usual congruence conditions: 
\begin{align*}
 \widehat\Z
 &:=\{(u_M)_{M\geq1}:u_M\in\Z/M\Z,\quad
                  u_N\equiv u_M\pmod M\text{ if }M\mid N\}.
\end{align*}
By the Chinese remainder theorem
\[
 \widehat\Z\cong\prod_p\Z_p.
\]
Here $\Z_p$ is the ring of $p$-adic integers.
Normalized Haar measure is characterized by
\begin{equation}\label{eq:Haar}
 \PP(U\equiv b\pmod M)=\frac1M.
\end{equation}
Thus the coordinates $U_p$ are independent, and an $M$-periodic
function $P$ satisfies
\[
 \E P(U)=\frac1M\sum_{b\bmod M}P(b).
\]
The translates $U+k$ use the same coordinates and are dependent.
We use $\E$, $\PP$, and $\Law$ for expectation, probability, and law.

\section{Periodic approximation of the multiplicities}\label{sec:arithmetic}

For $n\geq1$ and $M\geq1$, define
\begin{equation*}
 a(n):=\frac{r(n)}{2\pi\sqrt n},\qquad
 P_M(b):=\frac1{M^2}
 \#\{v\bmod M:|v|^2\equiv b\pmod M\}.
\end{equation*}
Thus $P_M$ is nonnegative, $M$-periodic, and has mean one.
For a sequence $b(n)$ define the mean-square seminorm
(the Besicovitch seminorm) by
\[
 \Bnorm{b}^{\,2}:=
 \limsup_{X\to\infty}\frac1X\sum_{1\leq n\leq X}|b(n)|^2.
\]

\begin{lemma}\label{conv:lem:B2}
Let $C:=4\pi^2/(21\zeta(3))$ and $M_Q:=\lcm(1,\ldots,Q)$.
Then
\begin{align}
 \Bnorm{a-P_M}^{\,2}
 &= C-\frac1M\sum_{b\bmod M}P_M(b)^2,\label{conv:eq:exactB}\\
 \Bnorm{a-P_{M_Q}}^{\,2}&\ll Q^{-1}.\label{conv:eq:Btail}
\end{align}
The limsup in \eqref{conv:eq:exactB} is a limit.
\end{lemma}

\begin{proof}
Set $A_2(x):=\sum_{1\leq n\leq x}r(n)^2$ and
$\kappa:=8\pi^4/(21\zeta(3))$. The input from \cite{CKO} is
\begin{equation}\label{conv:eq:CKO}
 A_2(x)=\kappa x^2+o(x^2).
\end{equation}
Partial summation gives
\begin{equation}\label{eq:partialsecond}
 \sum_{1\leq n\leq X}a(n)^2=CX+o(X),
 \qquad C=\frac{\kappa}{2\pi^2}.
\end{equation}

Lattice counting in a fixed congruence class gives
\begin{equation}\label{eq:latticeclass}
 \#\{x\in v+M\Z^3:|x|\leq R\}
 =\frac{4\pi R^3}{3M^3}+O_M(R^2+1).
\end{equation}
Summing over the classes representing $b$ and applying partial
summation yields
\begin{align}
 \sum_{\substack{0<n\leq X\\n\equiv b\ (M)}}r(n)
 &=\frac{4\pi P_M(b)}{3M}X^{3/2}+O_M(X+1),\label{eq:APr}\\
 \sum_{\substack{0<n\leq X\\n\equiv b\ (M)}}a(n)
 &=\frac{P_M(b)}M X+O_M(\sqrt X+1).\label{conv:eq:AP}
\end{align}
Consequently
\[
 \lim_{X\to\infty}\frac1X\sum_{n\leq X}a(n)P_M(n)
 =\lim_{X\to\infty}\frac1X\sum_{n\leq X}P_M(n)^2
 =\frac1M\sum_{b\bmod M}P_M(b)^2.
\]
Expanding the square and using \eqref{eq:partialsecond} proves
\eqref{conv:eq:exactB}, including existence of the limit.

Put $S(q,h):=\sum_{x\bmod q}e(hx^2/q)$. Character orthogonality gives
\[
 P_M(n)=M^{-3}\sum_{h\bmod M}S(M,h)^3e(-hn/M).
\]
Writing $h/M=h'/q$ in lowest terms and using
$S(M,h)=(M/q)S(q,h')$ gives
\begin{equation}\label{conv:eq:fourier}
  P_M(n)  =\sum_{q\mid M}\sum_{h\bmod q}^{*}
              c_{q,h}e(-hn/q),\qquad c_{q,h}:=q^{-3}S(q,h)^3.
\end{equation}
The star means $(h,q)=1$, with the constant term included at $q=1$.
Squaring out and again using orthogonality then gives
\begin{equation}\label{eq:finiteenergy}
 \frac1M\sum_{b\bmod M}P_M(b)^2
       =\sum_{q\mid M}\sum_{h\bmod q}^{*}|c_{q,h}|^2.
\end{equation}
The quadratic Gauss-sum evaluation (see \cite[pp.~18--28]{BEW}), for $(h,q)=1$, is
\[
 |S(q,h)|^2=\begin{cases}q,&q\text{ odd},\\
                         0,&q\equiv2\pmod4,\\
                         2q,&4\mid q.
           \end{cases}
\]
For each $q$, this magnitude is independent of the $\varphi(q)$
choices of $h$ coprime to $q$. Hence the total energy factors as
\begin{align}
 \sum_{q\geq1}\sum_{h\bmod q}^{*}|c_{q,h}|^2
 &=\sum_{\substack{q \ge 1\\q \text{ odd}}} \frac{\varphi(q)}{q^3} +8\sum_{\substack{q \ge 1\\4|q}} \frac{\varphi(q)}{q^3},\notag\\
 &=\left(1+\sum_{\nu\geq2}
              \frac{8\varphi(2^\nu)}{2^{3\nu}}\right)
               \sum_{q\ \mathrm{odd}}\frac{\varphi(q)}{q^3}\notag\\
 &=\frac43\cdot\frac67\frac{\zeta(2)}{\zeta(3)}
   =\frac{4\pi^2}{21\zeta(3)}=C.\label{eq:totalenergy}
\end{align}
Here $\sum_{q\geq1}\varphi(q)q^{-3}=\zeta(2)/\zeta(3)$ and
its factor at two is $1+\sum_{\nu\geq1}2^{\nu-1}2^{-3\nu}=7/6$.
Every $q\leq Q$ divides $M_Q$. Therefore
\[
 0\leq\Bnorm{a-P_{M_Q}}^2
 =\sum_{q\nmid M_Q}\sum_h^*|c_{q,h}|^2
 \leq8\sum_{q>Q}\frac{\varphi(q)}{q^3}
 \leq8\sum_{q>Q}q^{-2}\ll Q^{-1}.
\]
\end{proof}

\begin{remark}
The exact matching of the two second-moment constants is essential.
Knowing just an upper bound for $\sum r(n)^2$ would not prove that
the periodic approximation error tends to zero.
\end{remark}

\section{The periodic Green function}\label{sec:transfer}

The periodic Green function is
\begin{equation}\label{conv:eq:cot}
 (\cH_MP)(u,z):=-\frac\pi M\sum_{b=0}^{M-1}P(u+b)
                  \cot\left(\frac{\pi(z-b)}M\right),
 \qquad z\in\cD:=\C\setminus\Z,
\end{equation}
where $P$ is an $M$-periodic sequence and $u\in\Z/M\Z$.
For $P=P_M$, this is the scalar Green function of the periodic
arithmetic model. The weights $P_M(u+k)$ retain their arithmetic correlations.

To see where this function comes from, write $\lambda=m+z$ in the
spectral equation and use $r_3(n)=2\pi\sqrt n\,a(n)$ for $n\geq1$:
\[
 \frac{G(m+z)}{2\pi\sqrt m}
 =-\frac1{2\pi\sqrt m(m+z)}
 +\sum_{n\geq1}\sqrt{\frac nm}\,a(n)
       \left(\frac1{n-m-z}-\frac n{n^2+1}\right).
\]
Replace $a$ by $P_M$ and set $n=m+k$. For fixed $k$, the factor
$\sqrt{1+k/m}$ tends to one. This suggests the local model
\[
 \PV\sum_{k\in\Z}\frac{P_M(m+k)}{k-z}
       =(\cH_MP_M)(m,z).
\]
The cotangent identity and the passage to this principal value are
proved below. In particular, the regularized distant terms leave no
additional constant in the limit.

Write $P_Q:=P_{M_Q}$ and put
\[
 H_Q(u,z):=(\cH_{M_Q}P_Q)(u,z),\qquad
 F_m^{(c)}(z):=\frac{G(m+z)-c}{2\pi\sqrt m}.
\]
The main result of this section transfers the mean-square approximation
of the multiplicities to the spectral equation.

\begin{lemma}
\label{prop:Greenapprox}
For every compact $K\subset\cD$ and fixed $c\in\R$,
\begin{equation}\label{conv:eq:key}
 \limsup_{X\to\infty}\frac1X\sum_{X<m\leq2X}
       \knorm{F_m^{(c)}-H_Q(m,\cdot)}^2\leq C_KQ^{-1}.
\end{equation}
\end{lemma}

We prove three auxiliary lemmas before returning to the proof of
Lemma~\ref{prop:Greenapprox}. To transfer approximation errors, define
the linear operator
\begin{equation}\label{conv:eq:T}
 (T_zb)(m):=\frac1{\sqrt m}\sum_{n\geq1}\sqrt n\,b(n)
 \left(\frac1{n-m-z}-\frac n{n^2+1}\right).
\end{equation}
Thus
\[
 \frac{G(m+z)}{2\pi\sqrt m}
   =(T_za)(m)-\frac1{2\pi\sqrt m(m+z)},\qquad
 T_za-T_zP_Q=T_z(a-P_Q).
\]
For $\Bnorm b<\infty$, the series converges locally normally on
$\cD$ by Cauchy--Schwarz on dyadic blocks: we have
$\sum_{n\leq Y}|b(n)|^2=O_b(Y)$, and the regularized kernel is
$O_{m,K}(n^{-3/2})$.

\begin{lemma}\label{conv:lem:transfer}
For every compact $K\subset\cD$, there is a constant $C_K$,
independent of $b$, such that
\begin{equation}\label{conv:eq:Tbound}
 \limsup_{X\to\infty}\frac1X\sum_{X<m\leq2X}
 \sup_{z\in K}|(T_zb)(m)|^2
 \leq C_K\Bnorm{b}^{\,2}.
\end{equation}
\end{lemma}
\begin{proof}
Let $B:=\Bnorm b$ and fix $\varepsilon>0$. There is $Y_0$ such that
\begin{equation}\label{eq:prefixbound}
 \sum_{n\leq Y}|b(n)|^2\leq(B^2+\varepsilon)Y
 \qquad(Y\geq Y_0).
\end{equation}
Throughout, $m\in(X,2X]$, $z\in K$, and $X$ is sufficiently large
in terms of $K$. Split the sum into $n<X/2$, $X/2\leq n\leq4X$,
and $n>4X$.

\emph{The middle range.}
Set $f_X(n):=\sqrt n\,b(n)\ind_{[X/2,4X]}(n)$, extended by zero
to $\Z$. Writing $k=n-m$, decompose
\[
 \sum_n\frac{f_X(n)}{n-m-z}
 =\sum_{k\ne0}\frac{f_X(m+k)}k
   +\sum_{k\in\Z}r_z(k)f_X(m+k),
 \qquad
 r_z(k):=\begin{cases}
 -1/z,&k=0,\\
 z/[k(k-z)],&k\ne0.
 \end{cases}
\]
Since $K$ is compact and avoids $\Z$, the finitely many small
 denominators stay away from zero, while $r_z(k)=O_K(k^{-2})$
for large $|k|$. Thus $\sup_{z\in K}\|r_z\|_{\ell^1}<\infty$.
The discrete Hilbert transform and Young's inequality give
\[
 \sum_{m\in\Z}\left|\sum_n\frac{f_X(n)}{n-m-z}\right|^2
 \leq C_K\|f_X\|_{\ell^2}^2.
\]
Since $m>X$ and $n\leq4X$ on the support of $f_X$,
\begin{equation}\label{eq:middleH}
 \frac1X\sum_{X<m\leq2X}
 \left|\frac1{\sqrt m}\sum_n\frac{f_X(n)}{n-m-z}\right|^2
 \leq\frac{C_K}{X}\sum_{n\leq4X}|b(n)|^2
 \ll_K B^2+\varepsilon.
\end{equation}
The regularizing term is controlled by Cauchy--Schwarz:
\[
 \left|\frac1{\sqrt m}\sum_{X/2\leq n\leq4X}
             \frac{n^{3/2}b(n)}{n^2+1}\right|^2
 \leq\frac1X\sum_{n\leq4X}|b(n)|^2
             \sum_{X/2\leq n\leq4X}\frac1n
 \ll B^2+\varepsilon.
\]

\emph{The upper tail.}
Write the regularized kernel as
\[
 K_{m,z}(n):=\sqrt{\frac nm}
     \frac{1+n(m+z)}{(n-m-z)(n^2+1)}.
\]
For $n\asymp Y>4X$ it satisfies
$|K_{m,z}(n)|\leq C_K\sqrt X\,Y^{-3/2}$. Hence
\begin{align*}
 \sum_{n\asymp Y}|K_{m,z}(n)b(n)|
 &\leq C_K\sqrt X\,Y^{-3/2}\sqrt Y
                         \left(\sum_{n\asymp Y}|b(n)|^2\right)^{1/2}\\
 &\leq C_K(B^2+\varepsilon)^{1/2}(X/Y)^{1/2}.
\end{align*}
Summing over $Y=2^j4X$ gives $O_K((B^2+\varepsilon)^{1/2})$.
The subtraction must be combined with the resolvent in this range.

\emph{The lower tail.}
On $n\asymp Y<X/2$, its two kernels satisfy
\[
 \frac{\sqrt{n/m}}{|n-m-z|}\ll_KY^{1/2}X^{-3/2},\qquad
 \sqrt{\frac nm}\frac n{n^2+1}\ll X^{-1/2}Y^{-1/2}.
\]
By \eqref{eq:prefixbound}, the corresponding block sums are at most
\[
 C_K(B^2+\varepsilon)^{1/2}(Y/X)^{3/2},\qquad
 C(B^2+\varepsilon)^{1/2}(Y/X)^{1/2}.
\]
Their dyadic sums converge. The fixed set $n<Y_0$ contributes $o_b(1)$.
Together these estimates prove
\[
 \limsup_{X\to\infty}\frac1X\sum_{X<m\leq2X}|T_zb(m)|^2
 \leq C_K(B^2+\varepsilon).
\]

To pass to the supremum, cover $K$ by finitely many discs
$B(z_j,r_j)$ with $\overline{B(z_j,2r_j)}\subset\cD$. For any
holomorphic $f$ on the larger discs, the submean inequality gives
\[
 \sup_{z\in B(z_j,r_j)}|f(z)|^2
 \leq\frac1{\pi r_j^2}\int_{B(z_j,2r_j)}|f(w)|^2\,dA(w).
\]
Apply the uniform finite-$X$ estimates above under these area integrals,
then take $X\to\infty$ and $\varepsilon\downarrow0$.
\end{proof}

\begin{lemma}\label{lem:cotangent}
For every $M$-periodic sequence $P$,
\begin{equation}\label{eq:periodicPV}
 (\cH_MP)(u,z)=\lim_{L\to\infty}
                      \sum_{|k|\leq L}\frac{P(u+k)}{k-z}.
\end{equation}
Convergence is locally uniform for $z\in\cD$.
\end{lemma}
\begin{proof}
For $k=b+Mj$ and $\zeta:=(z-b)/M$, the cotangent partial-fraction
expansion \cite[Eq.~4.22.3]{DLMF} gives
\[
 \PV\sum_{j\in\Z}\frac1{b+Mj-z}
 =\frac1M\PV\sum_{j\in\Z}\frac1{j-\zeta}
 =-\frac\pi M\cot(\pi\zeta).
\]
The cutoff $|b+Mj|\leq L$ differs from a symmetric $j$-cutoff by
at most two terms, each $O_K(L^{-1})$. Summing the $M$ residue
classes proves the formula. Pairing positive and negative $j$ leaves
an $O_K(j^{-2})$ tail, which proves local uniformity.
\end{proof}

\begin{lemma}\label{conv:lem:periodic}
For every fixed $M$-periodic sequence $P$ and compact $K\subset\cD$,
\begin{equation}\label{conv:eq:periodic}
 \sup_{z\in K}|(T_zP)(m)-(\cH_MP)(m,z)|\longrightarrow0
 \quad(m\to\infty).
\end{equation}
In particular there is no additional real constant in this limit.
\end{lemma}
\begin{proof}
Write $w:=m+z$. The elementary identity
\[
 \sqrt n\left(\frac1{n-w}-\frac n{n^2+1}\right)
 =\frac w{\sqrt n(n-w)}+\frac1{\sqrt n(n^2+1)}
\]
gives
\begin{equation}\label{conv:eq:periodicsum}
 T_zP(m)=\left(1+\frac zm\right)
       \sum_{k\geq1-m}\frac{P(m+k)\sqrt{m/(m+k)}}{k-z}
       +O_P(m^{-1/2}).
\end{equation}
Split $P=\overline P+P_0$, where $P_0$ is periodic with mean zero.

Its interval sums satisfy
$\sup_{a\leq b}|\sum_{n=a}^bP_0(n)|\leq C_P$.
For $L>2\sup_K|z|+2$ and
$g_{m,z}(t):=\sqrt m/(\sqrt t(t-m-z))$,
\[
 g_{m,z}'(t)
 =-\frac{\sqrt m}{2t^{3/2}(t-m-z)}
  -\frac{\sqrt m}{\sqrt t(t-m-z)^2}.
\]
Splitting at $m/2$ and $2m$ gives the variation bound
\[
 |g_{m,z}(1)|+|g_{m,z}(m-L)|+|g_{m,z}(m+L)|
 +\int_{[1,m-L]\cup[m+L,\infty)}|g_{m,z}'(t)|\,dt
 \ll_Km^{-1/2}+L^{-1}.
\]
Summation by parts therefore shows that the weighted tails in
\eqref{conv:eq:periodicsum} are $O_{P,K}(L^{-1}+m^{-1/2})$.
The unweighted periodic tails are $O_{P,K}(L^{-1})$ by the same
argument with $(t-m-z)^{-1}$. On $|k|\leq L$,
\[
 \sqrt{m/(m+k)}=1+O(L/m).
\]
Take $m\to\infty$ and then $L\to\infty$ to obtain the required
limit for $P_0$.

It remains to treat $P=1$. On $|k|>L$, replacing $(k-z)^{-1}$
by $k^{-1}$ costs at most
\[
 C_K\sum_{\substack{n\geq1\\|n-m|>L}}
              \frac{\sqrt{m/n}}{(n-m)^2}
 \ll_K L^{-1}+m^{-1}.
\]
To control the signed tail, put
\[
 R_{m,L}:=\sum_{\substack{n\geq1\\|n-m|>L}}
                   \frac{\sqrt{m/n}}{n-m}.
\]
For fixed $0<\delta<1/2$, its portion with $|n/m-1|\geq\delta$
is a Riemann sum:
\[
 \sum_{|n/m-1|\geq\delta}\frac{\sqrt{m/n}}{n-m}
 \longrightarrow I_\delta:=
 \int_{(0,1-\delta]\cup[1+\delta,\infty)}
                         \frac{t^{-1/2}}{t-1}\,dt.
\]
The endpoint bounds near zero and infinity are respectively
$O(\sqrt\varepsilon)$ and $O(R^{-1/2})$, uniformly for the sums
and integrals. The intervening paired terms satisfy
\[
 \left|\sum_{L<k<\delta m}
       \frac{(1+k/m)^{-1/2}-(1-k/m)^{-1/2}}k\right|
 \leq C\delta.
\]
Finally,
\[
 \lim_{\delta\downarrow0}I_\delta
 =\PV\int_0^\infty\frac{2\,dv}{v^2-1}
 =\left[\log\left|\frac{v-1}{v+1}\right|\right]_0^\infty=0.
\]
Thus $\limsup_m|R_{m,L}|\leq|I_\delta|+C\delta$, so
$R_{m,L}\to0$ for fixed $L$. Restoring $z$ and the local terms,
then letting $L\to\infty$, gives
\[
 T_z1(m)\longrightarrow\PV\sum_{k\in\Z}\frac1{k-z}
                         =-\pi\cot(\pi z).
\]
All estimates are uniform for $z\in K$; no constant remains.
\end{proof}

\begin{proof}[Proof of Lemma~\ref{prop:Greenapprox}]
The definitions give the exact identity
\[
 F_m^{(c)}(z)
 =T_za(m)-\frac1{2\pi\sqrt m(m+z)}-\frac c{2\pi\sqrt m}.
\]
Therefore
\begin{align*}
 F_m^{(c)}(z)-H_Q(m,z)
 &=T_z(a-P_Q)(m)
   +\bigl(T_zP_Q(m)-H_Q(m,z)\bigr)\\
 &\hspace{1cm}-\frac1{2\pi\sqrt m(m+z)}-\frac c{2\pi\sqrt m}.
\end{align*}
The squared mean of the first term is $O_K(Q^{-1})$ by
Lemmas~\ref{conv:lem:B2} and \ref{conv:lem:transfer}.
For fixed $Q$, the second term tends uniformly to zero on $K$
by Lemma~\ref{conv:lem:periodic}. The final terms are
$O_K(m^{-3/2})$ and $O_c(m^{-1/2})$. Apply
$|x_1+\cdots+x_4|^2\leq4\sum_i|x_i|^2$ and take the limsup.
\end{proof}

\section{The limiting meromorphic function}\label{sec:random}

Let $\Omega:=\widehat\Z$ with normalized Haar measure,
and let $U$ be its coordinate random variable. 
For $M\mid N$, averaging over the $N/M$ lifts of a residue $b\pmod M$
gives
\begin{align}
 \E[P_N(U)\mid U\equiv b\pmod M]
 &=\frac MN\frac1{N^2}
   \#\{x\bmod N:|x|^2\equiv b\pmod M\}\notag\\
 &=\frac M{N^3}\left(\frac NM\right)^3 M^2P_M(b)
   =P_M(b).\label{conv:eq:mart}
\end{align}
Since $M_Q\mid M_{Q+1}$, this is a martingale identity. Its second
moments obey
\[
 \E P_Q(U)^2\leq C,\qquad
 \E|P_{Q'}(U)-P_Q(U)|^2
       =\E P_{Q'}(U)^2-\E P_Q(U)^2\quad(Q'\geq Q).
\]
The $L^2$ martingale convergence theorem
\cite[Theorem~4.4.6, with $p=2$]{Durrett} therefore gives a nonnegative limit $A$,
with
\begin{equation}\label{eq:Alimit}
 P_Q\longrightarrow A\text{ in }L^2\text{ and almost surely},\qquad
 \E A=1,\quad\E A^2=C,\quad\E|A-P_Q|^2\ll Q^{-1}.
\end{equation}

\begin{lemma}\label{conv:lem:positive}
Up to a Haar null set,
\begin{equation}\label{conv:eq:E}
 A(U)>0\quad\Longleftrightarrow\quad
 U_2\notin\bigcup_{a\geq0}4^a(7+8\Z_2).
\end{equation}
Call this event $E$ and set $I_k(U):=\ind_E(U+k)$.
Then $\PP(E)=5/6$. Simultaneously for all integer $k$,
$A(U+k)>0$ if and only if $I_k(U)=1$, almost surely.
\end{lemma}

\begin{proof}
The issue is strict positivity; nonnegative martingale convergence
alone does not settle it. We compute the local factors
$\beta_p(u):=\lim_{\nu\to\infty}P_{p^\nu}(u)$.

For odd $p$, let
\[
 N_p(u):=\#\{x\in(\Z/p\Z)^3:|x|^2\equiv u\pmod p\}
       =p^2P_p(u).
\]
Character orthogonality and the quadratic Gauss sum give
\[
 N_p(u)=p^2+p\left(\frac{-u}{p}\right)
 \quad(u\not\equiv0\pmod p),\qquad N_p(0)=p^2.
\]
Every nonzero solution modulo $p$ is nonsingular, and each such
solution has $p^{2(\nu-1)}$ lifts to a specified right-hand side
modulo $p^\nu$. When $p\mid u$, there are $p^2-1$ primitive
solutions. The imprimitive solutions have the form $x=py$; if
$p^2\mid u$ their normalized contribution is
\[
 p^{-2\nu}\,p^3\,
   \#\{y\bmod p^{\nu-2}:|y|^2\equiv u/p^2\pmod {p^{\nu-2}}\}
 =p^{-1}P_{p^{\nu-2}}(u/p^2).
\]
It follows that, for $u\ne0$,
\begin{equation}\label{eq:localrecurrence}
 \beta_p(u)=\begin{cases}
 1+p^{-1}(\frac{-u}{p}),&v_p(u)=0,\\
 1-p^{-2},&v_p(u)=1,\\
 1-p^{-2}+p^{-1}\beta_p(u/p^2),&v_p(u)\geq2.
 \end{cases}
\end{equation}
At $u=0$, the limiting value is $1+p^{-1}$. The same finite-level
recursion, and then bounded convergence, give
\begin{equation}\label{eq:localbounds}
 1-p^{-1}\leq\beta_p\leq1+p^{-1},\qquad
 \E\beta_p(U_p)=1,\qquad
 \E(\beta_p(U_p)-1)^2\leq p^{-2}.
\end{equation}

A similar recursive argument at the prime two yields, for
$u=4^a u_0$ with $4\nmid u_0$,
\begin{equation}\label{conv:eq:beta2}
 \beta_2(u)=2^{-a}
 \begin{cases}3/2,&u_0\equiv1,2\pmod4,\\
               1,&u_0\equiv3\pmod8,\\
               0,&u_0\equiv7\pmod8.
 \end{cases}
 \qquad \beta_2(0)=0.
\end{equation}
Finite densities are bounded by $3/2$, so
$\E\beta_2(U_2)=1$ and
$\E[\beta_p(U_p)\mid U_p\bmod p^\nu]=P_{p^\nu}(U_p)$ for all $p$.

The odd-prime product $W:=\prod_{p>2}\beta_p(U_p)$ converges in
$L^2$, since
\[
 \E\prod_{2<p\leq Y}\beta_p^2
 =\prod_{2<p\leq Y}\bigl(1+O(p^{-2})\bigr)\ll1.
\]
For strict positivity, Taylor expansion and independence give
\[
 \sum_{p>2}|\E\log\beta_p|<\infty,\qquad
 \sum_{p>2}\Var(\log\beta_p)<\infty.
\]
The centered logarithmic partial sums are an $L^2$-bounded
martingale. Their almost-sure convergence shows $0<W<\infty$
almost surely. By the Chinese remainder theorem,
\[
 B(U):=\beta_2(U_2)W,\qquad
 \E[B\mid U\bmod M]=P_M(U).
\]
The congruence sigma-fields generate the Borel sigma-field of
$\widehat\Z$; hence $A=B$ in $L^2$ and almost surely.
Formula \eqref{conv:eq:beta2} now proves the assertion, apart from
the null set $U_2=0$. Finally,
\[
 \PP(A=0)=\sum_{a\geq0}\frac1{8\cdot4^a}=\frac16.
\]
Translation invariance and a countable union of null sets give the
simultaneous assertion for every $k\in\Z$.
\end{proof}

\begin{lemma}\label{conv:lem:cyclic}
Let $M\geq1$ and let $b$ be an $M$-periodic complex sequence. For
every compact $K\subset\cD$,
\begin{equation}\label{conv:eq:cyclic}
 \frac1M\sum_{u\bmod M}\knorm{\cH_Mb(u,\cdot)}^2
 \leq\frac{C_K}{M}\sum_{u\bmod M}|b(u)|^2,
\end{equation}
where $C_K$ is independent of $M$ and $b$. If $M\mid N$ and $b$ is
$M$-periodic, then $\cH_Nb=\cH_Mb$.
\end{lemma}

\begin{proof}
Recall from \eqref{conv:eq:cot} that $\cH_M$ acts on $M$-periodic
sequences, and that Lemma~\ref{lem:cotangent} gives the principal-value
representation
\[
 (\cH_Mb)(u,z)=\PV\sum_{k\in\Z}\frac{b(u+k)}{k-z}.
\]
We first prove an $L^2$ estimate for this operator.  Let
\[
 \chi_{r,M}(u):=e(ru/M),\qquad
 \widehat b(r):=\frac1M\sum_{u\bmod M}b(u)e(-ru/M),
 \qquad 0\leq r<M.
\]
Thus $b(u)=\sum_{r=0}^{M-1}\widehat b(r)\chi_{r,M}(u)$.  Applying
$\cH_M$ to a character and using Lemma~\ref{lem:cotangent}, we find
\[
 (\cH_M\chi_{r,M})(u,z)
 =\chi_{r,M}(u)\mathfrak m_z(r/M),
 \qquad
 \mathfrak m_z(\theta):=\PV\sum_{k\in\Z}
                         \frac{e(k\theta)}{k-z}.
\]
Each
Fourier character is an eigenfunction of $\cH_M$, and the multiplier $\mathfrak m_z(\theta)$
is the scalar by which that character is multiplied.

We record the scalar explicitly.  For $r=0$, the cotangent
partial-fraction formula gives $\mathfrak m_z(0)=-\pi\cot(\pi z)$.
If $1\leq r<M$, then \eqref{conv:eq:cot} gives
\[
 \mathfrak m_z(r/M)
 =-\frac\pi M\sum_{a=0}^{M-1}e(ra/M)
       \cot\left(\frac{\pi(z-a)}M\right).
\]
To evaluate the finite sum, set $\omega:=e(1/M)$ and
$\zeta:=e(z/M)$.  Writing the cotangent in exponential form
and summing a geometric progression gives
\[
 \sum_{a=0}^{M-1}\omega^{ra}
       \cot\left(\frac{\pi(z-a)}M\right)
 =2i\sum_{a=0}^{M-1}\frac{\omega^{(r+1)a}}{\zeta-\omega^a}
 =\frac{2iM\zeta^r}{\zeta^M-1}.
\]
Since $\zeta^M-1=2ie^{\pi iz}\sin(\pi z)$, this proves
\begin{equation}\label{eq:meromorphicmultiplier}
 \mathfrak m_z(\theta)=\PV\sum_{k\in\Z}\frac{e(k\theta)}{k-z}
 =\begin{cases}
 -\pi\cot(\pi z),&\theta=0,\\
 -\pi e^{i\pi(2\theta-1)z}/\sin(\pi z),&0<\theta<1.
 \end{cases}
\end{equation}
Here and below $\theta$ is represented in $[0,1)$.  If $K'$ is a
compact subset of $\cD$, then $\sin(\pi z)$ is bounded away from zero
on $K'$, while
$|e^{i\pi(2\theta-1)z}|\leq e^{\pi|\Im z|}$.  Consequently
\[
 \sup_{0\leq\theta<1}\sup_{z\in K'}
       |\mathfrak m_z(\theta)|<\infty.
\]

Fourier inversion and \eqref{eq:meromorphicmultiplier} now give
\[
 (\cH_Mb)(u,z)=\sum_{r=0}^{M-1}\widehat b(r)
   \mathfrak m_z(r/M)\chi_{r,M}(u).
\]
For each fixed $z\in K'$, Parseval therefore implies
\begin{align*}
 \frac1M\sum_{u\bmod M}|(\cH_Mb)(u,z)|^2
 &=\sum_{r=0}^{M-1}|\widehat b(r)|^2
                    |\mathfrak m_z(r/M)|^2\\
 &\leq C_{K'}\sum_{r=0}^{M-1}|\widehat b(r)|^2
 =\frac{C_{K'}}M\sum_{u\bmod M}|b(u)|^2.
\end{align*}
It remains to make this estimate uniform in $z\in K$.  Cover $K$ by
finitely many discs $B(z_j,r_j)$ such that
$\overline{B(z_j,2r_j)}\subset\cD$, and put
$K':=\bigcup_j\overline{B(z_j,2r_j)}$. If $f$ is holomorphic on the larger discs,
the submean inequality gives
\[
 \sup_{z\in B(z_j,r_j)}|f(z)|^2
 \leq\frac1{\pi r_j^2}
       \int_{B(z_j,2r_j)}|f(w)|^2\,dA(w).
\]
Apply this to $f(z)=(\cH_Mb)(u,z)$, sum over $j$, average over
$u\bmod M$, and use the fixed-$w$ Parseval estimate under each area
integral.  We obtain
\[
 \frac1M\sum_{u\bmod M}\knorm{\cH_Mb(u,\cdot)}^2
 \leq\frac{C_K}{M}\sum_{u\bmod M}|b(u)|^2.
\]
If $M\mid N$ and $b$ is $M$-periodic, Lemma~\ref{lem:cotangent}
represents both transforms by the same symmetric principal value.
This proves the compatibility assertion.
\end{proof}

\begin{lemma}\label{conv:lem:H}
There is a random function $H(U,\cdot)\in\cO(\cD)$ such that, for
every compact $K\subset\cD$,
\begin{equation}\label{conv:eq:Hconv}
 \E\knorm{H_Q(U,\cdot)-H(U,\cdot)}^2\longrightarrow0.
\end{equation}
Almost surely, $H$ extends meromorphically to $\C$, is real on the
real axis away from its poles, and has exactly the poles
\[
 \{k\in\Z:I_k(U)=1\},\qquad
 \Res_{z=k}H(U,z)=-A(U+k).
\]
Between consecutive poles it is strictly increasing, with one simple
real zero in each such interval.
\end{lemma}

\begin{proof}
If $Q'\geq Q$, then $M_Q\mid M_{Q'}$, and the compatibility in
Lemma~\ref{conv:lem:cyclic} gives
\[
 H_{Q'}(U,\cdot)-H_Q(U,\cdot)
 =\cH_{M_{Q'}}(P_{Q'}-P_Q)(U,\cdot).
\]
The reduction of Haar measure modulo $M_{Q'}$ is uniform, so
\eqref{conv:eq:cyclic} and \eqref{eq:Alimit} give
\[
 \E\knorm{H_{Q'}-H_Q}^2
 \leq C_K\E|P_{Q'}-P_Q|^2
 \leq 2C_K\bigl(\E|P_{Q'}-A|^2+\E|P_Q-A|^2\bigr)
 \ll_KQ^{-1}.
\]
Thus $(H_Q)$ is Cauchy in $L^2(\Omega;C(K))$ for every compact
$K\subset\cD$.  Choose a nested compact exhaustion of $\cD$ and use
the completeness of $C(K)$ on each member of the exhaustion.  The
resulting compatible limits define a random holomorphic function
$H\in\cO(\cD)$ and give \eqref{conv:eq:Hconv}.  A diagonal
subsequence converges almost surely locally uniformly on $\cD$.
After intersecting with a further probability-one event, we also have
\[
 P_Q(U+k)\longrightarrow A(U+k)\qquad(k\in\Z),
\]
simultaneously for all integers $k$, by \eqref{eq:Alimit} and
translation invariance.

We work on this probability-one event.  Fix $k\in\Z$ and
$0<r<1/2$.  The functions
\[
 J_{Q,k}(z):=H_Q(U,z)+\frac{P_Q(U+k)}{z-k}
\]
are holomorphic on $|z-k|<r$: the second term cancels the only pole of
$H_Q$ in that disc.  On the boundary circle, local uniform convergence
of $H_Q$ and convergence of $P_Q(U+k)$ show that $J_{Q,k}$ converges
uniformly.  Cauchy's integral formula then gives a holomorphic limit
$J_k$ on the disc and uniform convergence on every smaller closed
disc.  On the punctured disc we therefore have
\[
 H(U,z)=-\frac{A(U+k)}{z-k}+J_k(z),\qquad J_k\text{ holomorphic}.
\]
Thus $z=k$ is a pole precisely when $A(U+k)>0$, and in that case its
residue is $-A(U+k)$.  Lemma~\ref{conv:lem:positive} identifies these
integers with $I_k(U)=1$. Since $H_Q(U,x)$ are real valued functions for $x\in \R\setminus \Z$, $H(U,x)$ is real too.

For $x\in\R\setminus\Z$ and any surviving pole $k$,
\[
 H_Q'(U,x)=\sum_{j\in\Z}\frac{P_Q(U+j)}{(j-x)^2}
 \geq\frac{P_Q(U+k)}{(k-x)^2}.
\]
The differentiated series is absolutely convergent.  Local
holomorphic convergence permits differentiation, and passage to the
limit gives
\[
 H'(U,x)\geq\frac{A(U+k)}{(k-x)^2}>0.
\]
By continuity the derivative remains positive at every removable
integer point.  If $k<l$ are consecutive poles, the residue
description gives
\[
 H(x)\longrightarrow-\infty\ (x\downarrow k),\qquad
 H(x)\longrightarrow+\infty\ (x\uparrow l),\qquad H'(x)>0.
\]
The intermediate value theorem proves existence and uniqueness of
the zero, and the derivative inequality proves its simplicity.
\end{proof}

\section{Convergence of the spacing measures}\label{sec:zero}

On the event $E$, let
\[
 k_-:=\max\{k<0:I_k(U)=1\},\qquad
 k_+:=\min\{k>0:I_k(U)=1\}.
\]
The residue-class argument proving \eqref{conv:eq:bounded}
applies equally over $\Z_2$: $k_-\geq-3$, $k_+\leq3$ and
$k_+-k_-\leq4$. Let $Z_-(U)$ and $Z_+(U)$ be the unique
zeros of $H(U,\cdot)$ in $(k_-,0)$ and $(0,k_+)$.
Define
\begin{equation}\label{conv:eq:law}
 \mu:=\Law\left(\rho\bigl(Z_+(U)-Z_-(U)\bigr)\,\middle|\,E\right).
\end{equation}
This characterizes the limiting law entirely in terms of congruence
counts, cotangents, and Haar measure.

\begin{lemma}\label{conv:lem:joint}
If $m$ is uniformly sampled from the integers in $(X,2X]$, then
\begin{equation}\label{conv:eq:joint}
 \left(F_m^{(c)},\bigl(\ind_{\cS}(m+k)\bigr)_{-3\leq k\leq3}\right)
 \ \Longrightarrow\
 \left(H(U,\cdot),\bigl(I_k(U)\bigr)_{-3\leq k\leq3}\right)
\end{equation}
in $\cO(\cD)$ with the compact-open topology, times a finite
discrete space.
\end{lemma}

\begin{proof}
Equip $\cO(\cD)$ with the bounded compact-open metric
\[
 d(f,g):=\sum_{j\geq1}2^{-j}\min(1,\|f-g\|_{K_j}),
\]
where $(K_j)$ is a cofinal compact exhaustion. Truncating this
series and applying Lemma~\ref{prop:Greenapprox} gives
\begin{equation}\label{eq:metricapprox}
 \lim_{Q\to\infty}\limsup_{X\to\infty}
 \E_X d(F_m^{(c)},H_Q(m,\cdot))=0.
\end{equation}
Here $\E_X$ denotes uniform integer sampling in $(X,2X]$.

For the pole indicators, put
\[
 J_R(u):=1-\sum_{\ell=0}^{R-1}\ind_{4^\ell(7+8\Z_2)}(u),
 \qquad L_R:=2^{2R+1}.
\]
The function $J_R$ is $L_R$-periodic. The omitted exceptional
classes have Haar mass $4^{-R}/6$; counting these same progressions
in an interval gives, uniformly for $|k|\leq3$,
\begin{align*}
 \PP(J_R(U+k)\ne I_k(U))&=4^{-R}/6,\\
 \PP_X(J_R(m+k)\ne\ind_{\cS}(m+k))
 &\ll4^{-R}+X^{-1}\log(2X).
\end{align*}
The second estimate uses only $O(\log X)$ progressions, each
contributing its expected count with an $O(1)$ error.

Choose $R(Q)\to\infty$ with $L_{R(Q)}\mid M_Q$. On the Polish space
\[
 \mathscr X:=\cO(\cD)\times\{0,1\}^7,
 \qquad \Delta((f,i),(g,j)):=d(f,g)+\ind_{\{i\ne j\}},
\]
define
\begin{align*}
 Y_X(m)&:=\left(F_m^{(c)},(\ind_{\cS}(m+k))_{|k|\leq3}\right),\\
 Y_Q(u)&:=\left(H_Q(u,\cdot),(J_{R(Q)}(u+k))_{|k|\leq3}\right),\\
 Y(U)&:=\left(H(U,\cdot),(I_k(U))_{|k|\leq3}\right).
\end{align*}
Equation~\eqref{eq:metricapprox}, the indicator estimates, and
Lemma~\ref{conv:lem:H} imply
\[
 \lim_{Q\to\infty}\limsup_{X\to\infty}
       \E_X\Delta(Y_X,Y_Q(m))=0,
 \qquad
 \E\Delta(Y_Q(U),Y(U))\longrightarrow0.
\]
For fixed $Q$, equidistribution of integers modulo $M_Q$ gives
$Y_Q(m)\Rightarrow Y_Q(U)$ as $X\to\infty$. Testing against
bounded Lipschitz functions and using the two approximation bounds
proves \eqref{conv:eq:joint}.
\end{proof}

\begin{lemma}\label{conv:lem:zeros}
Suppose meromorphic functions $f_j$ with poles only at integers
converge locally uniformly on $\cD$ to $f$. Suppose each function
is real and strictly increasing between its actual poles, with
one zero in each interval between consecutive poles. Suppose
also that the pole indicators on $[-3,3]\cap\Z$ eventually
agree with those of $f$, that $0$ is a pole of $f$, and that
the neighboring poles of $f$ lie in $[-3,3]$.
If all actual residues of $f$ are strictly negative, the zeros
of $f_j$ immediately to the left and right of $0$ converge
to the corresponding zeros of $f$.
\end{lemma}

\begin{proof}
The neighboring poles eventually equal fixed integers $k_-<0<k_+$.
Any integer inside $(k_-,0)\cup(0,k_+)$ is eventually removable
for every $f_j$. For a small disc about such an integer, uniform
convergence on its boundary and the maximum principle give uniform
convergence throughout the disc. Thus $f_j\to f$ locally uniformly
on both real intervals, including their removable integer points.

Let $z_-$ be the zero in $(k_-,0)$. For every sufficiently small
$\varepsilon>0$,
\[
 f(z_--\varepsilon)<0<f(z_-+\varepsilon).
\]
Uniform convergence at these two points gives the same signs for
$f_j$ when $j$ is large. Its unique zero $z_{-,j}$ must therefore satisfy
\[
 z_--\varepsilon<z_{-,j}<z_-+\varepsilon.
\]
Letting $\varepsilon\downarrow0$ proves $z_{-,j}\to z_-$.
The identical argument on $(0,k_+)$ gives $z_{+,j}\to z_+$.
\end{proof}

\begin{proof}[Proof of Theorem~\ref{conv:thm:main}]
Fix $f\in C([0,10/3])$.  We first average over a dyadic interval of
integers, without conditioning the central integer to lie in $\cS$.
Take any sequence of positive integers $X_n\to\infty$. By
Lemma~\ref{conv:lem:joint} and the Skorokhod representation theorem
\cite[Theorem~6.7]{Billingsley}, copies
\[
 Y_n=(F_n,(\varepsilon_{n,k})_{|k|\leq3}),
 \qquad Y=(H,(I_k(U))_{|k|\leq3})
\]
may be realized on one probability space so that $Y_n\to Y$ almost
surely. For each $n$ with $\varepsilon_{n,0}=1$, let $Z_{n,-}$ and
$Z_{n,+}$ be the zeros of $F_n$ adjacent to the origin; when
$\varepsilon_{n,0}=0$, set $Z_{n,-}=Z_{n,+}=0$. Likewise set
$Z_-(U)=Z_+(U)=0$ off $E$. These variables are measurable:
on an interval $(a,b)$ between marked poles, the unique zero of an
increasing function $g$ is
\[
 \inf\{r\in(\Q\setminus\Z)\cap(a,b):g(r)>0\}.
\]

The discrete coordinates are eventually equal. On $I_0(U)=1$,
Lemma~\ref{conv:lem:zeros} gives
\[
 Z_{n,-}\longrightarrow Z_-(U),\qquad
 Z_{n,+}\longrightarrow Z_+(U),
\]
while on $I_0(U)=0$ all four quantities are eventually zero. Hence
\[
 \varepsilon_{n,0}f\bigl(\rho(Z_{n,+}-Z_{n,-})\bigr)
 \longrightarrow
 I_0(U)f\bigl(\rho(Z_+(U)-Z_-(U))\bigr)
\]
almost surely. The variables are bounded by $\|f\|_\infty$, so their
expectations converge. Since the coupled variables retain their
marginal laws, these expectations are
\begin{align*}
 \lim_{n\to\infty}\frac1{X_n}
   \sum_{\substack{X_n<m\leq2X_n\\m\in\cS}}f(s_m(c))
 &=\E\left[\ind_E(U)
   f\bigl(\rho(Z_+(U)-Z_-(U))\bigr)\right]
 =\rho\int f\,d\mu.
\end{align*}
The last equality uses $\PP(E)=\rho$ and \eqref{conv:eq:law}.
The sequence $(X_n)$ was arbitrary, and changing an endpoint by
$O(1)$ contributes $o(1)$. Hence
\begin{equation}\label{conv:eq:dyadic}
 \frac1X\sum_{\substack{X<m\leq2X\\m\in\cS}}f(s_m(c))
 \longrightarrow\rho\int f\,d\mu
 \qquad(f\in C([0,10/3])).
\end{equation}

Summing dyadically and letting $N(X)/X\to\rho$ then gives
\begin{equation}\label{eq:poleaverage}
 \frac1{N(X)}\sum_{\substack{0<m\leq X\\m\in\cS}}f(s_m(c))
 \longrightarrow\int f\,d\mu.
\end{equation}

Taking $X=m_N$ in \eqref{eq:poleaverage} then gives
\[
 \int f\,d\mu_N^{(c)}=\frac1N\sum_{j=1}^Nf(s_j(c))
 \longrightarrow\int f\,d\mu,
\]
which proves \eqref{conv:eq:limit}.  The limiting function in
Lemma~\ref{conv:lem:joint} is the same for every fixed $c$: in the
normalized spectral equation the only $c$-dependent term is
$-c/(2\pi\sqrt m)$, which tends to zero.  Thus the measure $\mu$
is independent of the interaction parameter.

The remaining properties of $\mu$ follow more or less directly from the preceding work.
\end{proof}

\section{Level repulsion}\label{sec:repulsion}

Section~\ref{sec:zero} identifies the limiting measure with the two
zeros of $H(U,\cdot)$ adjacent to a pole at the origin. On the event
$E:=\{A(U)>0\}$, write these zeros as $-u(U)<0<v(U)$ and set
$d(U):=u(U)+v(U)$. Before unfolding, the small-gap distribution is
\begin{equation}\label{eq:law}
 F_{\rm raw}(s):=\PP(d(U)\leq s\mid E),\qquad
 \mu([0,s])=F_{\rm raw}(s/\rho).
\end{equation}
We shall determine the asymptotic behavior of $F_{\rm raw}(s)$ as
$s\downarrow0$.

We first derive Theorem~\ref{thm:spacing} from the required density
estimate and then prove that estimate. Its strict positivity is proved
in Section~\ref{sec:positivity}.

\subsection{The local Green function and small gaps}\label{sec:conditional}

The factorization $\widehat\Z=\Z_2\times\prod_{p>2}\Z_p$ separates the
two parts of the argument. Write $U=(t,V)$, where $t\in\Z_2$ and
$V=(V_p)_{p>2}$ is Haar on
\[
 \Omega_o:=\prod_{p>2}\Z_p.
\]
By Lemma~\ref{conv:lem:positive}, the limiting residue field factors as
\[
 A(U+k)=\beta_2(t+k)\prod_{p>2}\beta_p(V_p+k).
\]
The odd-prime product is positive almost surely. Thus the event $E$ is
exactly $\{\beta_2(t)>0\}$, while the size of the central residue is
governed by the $2$-adic coordinate. To keep the odd-prime randomness
visible, put
\begin{equation}\label{eq:W}
 W(V):=\prod_{p>2}\beta_p(V_p),\qquad
 A_t(k):=\beta_2(t+k)W(V+k),\qquad R_t:=A_t(0).
\end{equation}
Here $R_t$ is the residue size at the central pole.

We next define the two coefficients that remain after the pole at the
origin is removed:
\begin{equation}\label{eq:hD}
 h_t:=\PV\sum_{k\ne0}\frac{A_t(k)}k,
 \qquad
 D_t:=\sum_{k\ne0}\frac{A_t(k)}{k^2}.
\end{equation}
The principal value is the symmetric limit in $L^2(\Omega_o)$, and
$D_t>0$ almost surely. Separating the central pole gives
\begin{equation}\label{eq:H}
 H_t(z):=-\frac{R_t}{z}+h_t+
 \sum_{k\ne0}A_t(k)\left(\frac1{k-z}-\frac1k\right).
\end{equation}
For Haar-almost every $(t,V)$ this is precisely the function
$H(U,z)$ constructed in Section~\ref{sec:random}. Indeed, the identity
holds for the periodic approximants; \eqref{eq:Alimit} and the operator
bounds in Lemma~\ref{lem:Lq} permit passage to the limit. Normal
convergence away from $\Z$ follows from the summable regularized kernel.

The term $-R_t/z$ is the singular part of $H_t$ at the origin.  After
removing it, the function
\[
 F_t(z):=H_t(z)+\frac{R_t}{z}
\]
extends analytically across $z=0$.  We call $F_t$ the regular part of
$H_t$ at the origin.  In particular,
\[
 F_t(0)=h_t,\qquad F_t'(0)=D_t.
\]
Positivity of the coefficients gives, uniformly for $|x|\leq1/2$,
\begin{equation}\label{eq:localTaylor}
 F_t'(x)=D_t(1+O(|x|)),\qquad
 F_t(x)=h_t+D_tx+O(D_tx^2).
\end{equation}
Thus $R_t,h_t,D_t$ are not auxiliary parameters: they are exactly the
pole strength, the regular value, and the regular slope that determine
the two nearby zeros to second order.

On $E$, let $-u_t<0<v_t$ be the adjacent zeros of $H_t$ and put
$d_t:=u_t+v_t$. The following elementary calculation turns the
small-gap problem into an inequality involving these three local
coefficients.

\begin{lemma}\label{lem:quadratic}
For $R_t>0$ set
\begin{equation}\label{eq:quadratic}
 q_t:=\frac{\sqrt{h_t^2+4D_tR_t}}{D_t}.
\end{equation}
There exist $\varepsilon_0,C>0$ such that
\[
 q_t\leq\varepsilon_0
 \quad\Longrightarrow\quad
 |d_t-q_t|\leq Cq_t^2,
\]
and
\[
 d_t\leq\varepsilon_0
 \quad\Longrightarrow\quad
 |q_t-d_t|\leq Cd_t^2.
\]
\end{lemma}

\begin{proof}
Suppress the subscript $t$ and write $R,h,D,H,F,u,v,d,q$ for the
quantities above.  Thus
\[
 H(x)=-\frac{R}{x}+F(x),\qquad
 F(x)=h+Dx+O(Dx^2),\qquad
 F'(x)=D(1+O(|x|)).
\]

We first assume that $q$ is small.  Replace $F(x)$ by its linear
approximation $h+Dx$ and consider
\[
 \widetilde H(x):=-\frac R{x}+h+Dx.
\]
Its two roots are
\[
 x_\pm=\frac{-h\pm\sqrt{h^2+4DR}}{2D},
 \qquad x_-<0<x_+,\qquad x_+-x_-=q.
\]
Because the roots have opposite signs, $|x_\pm|\leq q$.  Moreover,
\[
 H(x)=\widetilde H(x)+O(Dx^2),
\]
so that
\[
 H(x_\pm)=O(Dx_\pm^2).
\]
For $|x|$ sufficiently small,
\[
 H'(x)=\frac{R}{x^2}+F'(x)\geq \frac{D}{2}.
\]
Moving a distance $C_1x_\pm^2$ away from $x_\pm$ therefore changes
$H$ by at least $(D/2)C_1x_\pm^2$.  If $C_1$ is sufficiently large,
this change dominates the error $H(x_\pm)=O(Dx_\pm^2)$.  Hence $H$
has opposite signs at the two endpoints
$x_\pm-C_1x_\pm^2$ and $x_\pm+C_1x_\pm^2$.  The true zeros
$z_-=-u$ and $z_+=v$ consequently satisfy
\[
 z_\pm=x_\pm+O(x_\pm^2).
\]
Therefore
\[
 d=z_+-z_-
   =(x_+-x_-)+O(x_+^2+x_-^2)
   =q+O(q^2).
\]

We now prove the converse.  Assume that $d=u+v$ is small.  Since
$-u$ and $v$ are zeros of $H$, we have
\[
 F(v)=\frac{R}{v},\qquad F(-u)=-\frac{R}{u}.
\]
Subtracting these identities gives
\[
 F(v)-F(-u)
 =R\left(\frac1v+\frac1u\right)
 =\frac{Rd}{uv}.
\]
On the other hand, the derivative estimate for $F$ gives
\[
 F(v)-F(-u)
 =\int_{-u}^{v}F'(x)\,dx
 =Dd+O(Dd^2).
\]
Cancelling $d$ from the last two displays yields
\[
 R=Duv(1+O(d)).
\]
Next, the identity $F(v)=R/v$ and the Taylor expansion of $F$ give
\[
 h+Dv+O(Dv^2)=\frac{R}{v}=Du(1+O(d)).
\]
Since $u,v\leq d$, this implies
\[
 h=D(u-v)+O(Dd^2).
\]
Consequently,
\begin{align*}
 h^2+4DR
 &=D^2(u-v)^2+4D^2uv+O(D^2d^3)\\
 &=D^2(u+v)^2+O(D^2d^3)\\
 &=D^2d^2(1+O(d)).
\end{align*}
Here we used the elementary identity
$(u-v)^2+4uv=(u+v)^2$.  Taking the positive square root and dividing
by $D$ now gives
\[
 q=d\sqrt{1+O(d)}=d+O(d^2),
\]
as required.
\end{proof}

For $s>0$, define
\[
 \mathcal Q_t(s):=\{q_t\leq s\}.
\]
Since $D_t>0$, the definition of $q_t$ gives the exact identity
\begin{equation}\label{eq:quadratic-event}
 \mathcal Q_t(s)
 =\left\{
 R_t\leq\frac{s^2D_t}{4},\quad
 |h_t|\leq\bigl(s^2D_t^2-4D_tR_t\bigr)^{1/2}
 \right\}.
\end{equation}
In particular,
\[
 \mathcal Q_t(s)\subseteq
 \left\{R_t\leq\frac{s^2D_t}{4}\right\}
 \cap\{|h_t|\leq sD_t\}.
\]
Thus the residue $R_t$ must be of order at most $s^2D_t$, and the
regular value $h_t$ must lie in an interval of length at most
$2sD_t$ around the origin.

We now rewrite the residue inequality in terms of the $2$-adic
coordinate. Conditional on $\beta_2(t)>0$, there are unique
$\ell\geq0$ and $b\in\{1,3/2\}$ such that
\[
 t=4^\ell t_0,\qquad 4\nmid t_0,\qquad b:=\beta_2(t_0).
\]
The explicit formula \eqref{conv:eq:beta2} gives
\begin{equation}\label{eq:marks}
 \PP(\ell=n)=\tfrac34\,4^{-n}\quad(n\geq0),
 \qquad
 \PP(b=1)=\tfrac15,\quad \PP(b=3/2)=\tfrac45.
\end{equation}
The variables $\ell,b,V$ are independent, and
\begin{equation}\label{eq:residue}
 R_t=2^{-\ell}bW.
\end{equation}
Substitution into \eqref{eq:quadratic-event} yields
\[
 \mathcal Q_t(s)=
 \left\{
 2^{-\ell}bW\leq\frac{s^2D_t}{4},\quad
 |h_t|\leq
 \bigl(s^2D_t^2-4D_t2^{-\ell}bW\bigr)^{1/2}
 \right\}.
\]
For an integer $n$ satisfying $2^{-n}\asymp s^2$, we have
\[
 \PP(\ell=n)=\tfrac34 4^{-n}\asymp s^4.
\]
For each such index, the permitted interval for $h_t$ has length
\[
 2\bigl(s^2D_t^2-4D_t2^{-\ell}bW\bigr)^{1/2}
 \leq 2sD_t.
\]
Corollary~\ref{prop:density} controls this interval jointly with
$D_t$ and $W$ and yields the remaining factor $s$.  More precisely, if
\[
 \log_2(s^{-2})=L+\alpha,\qquad L\in\Z,\quad 0\leq\alpha<1,
\]
and $\ell=L+j$, then
\[
 2^{-\ell}=s^2 2^{\alpha-j},\qquad
 4^{-\ell}=s^4 2^{2(\alpha-j)}.
\]
The dependence on the fractional part $\alpha$ produces the periodic
factor in the final asymptotic.

\subsection{From smoothing to the small-gap law}\label{sec:traces}\label{sec:spacing}

To evaluate \eqref{eq:quadratic-event}, we need the distribution of
$h_t$ in an interval of length $s$, jointly with the variables $D_t$
and $W$. Moreover, on the relevant classes $\ell\to\infty$, so
$t\to0$ in $\Z_2$. The needed input is therefore a density in the
$h_t$ variable that is continuous as $(t,h_t)\to(0,0)$, with enough
uniform moment control in $D_t$ and $W$ to sum over all dyadic classes.

Put
\begin{equation}\label{eq:omega}
 \omega(D,W):=1+D+D^{-1}+W+W^{-1},\qquad (D,W)\in(0,\infty)^2.
\end{equation}
Call $\psi$ admissible if $\psi\in C^2((0,\infty)^2)$ and, for some
finite $C,m$,
\begin{equation}\label{eq:amplitude}
 |\partial^\gamma\psi(D,W)|\leq C\omega(D,W)^m
 \quad\text{for every }|\gamma|\leq2.
\end{equation}
This includes smooth compactly supported tests, constants, and powers
of $D,W$ and their reciprocals.

\begin{proposition}\label{thm:smoothing}
For every admissible $\psi$,
\begin{equation}\label{eq:weightedchar}
 \sup_{t\in\Z_2}
 \left|\E_V\bigl[\psi(D_t,W)e^{i\xi h_t}\bigr]\right|
 \ll_\psi(1+|\xi|)^{-3/2}.
\end{equation}
\end{proposition}

The role of the weight $\psi$ is to retain the dependence of
$h_t$ on $D_t$ and $W$; an unweighted density for $h_t$ would not
control the region in \eqref{eq:quadratic-event}. Only integrability
of the Fourier bound is needed. We first use the proposition to finish
the small-gap calculation. Its proof occupies
Sections~\ref{sec:local}--\ref{sec:retain}.

\begin{corollary}\label{prop:density}
There are finite positive measures $\nu_{t,y}$ on $(0,\infty)^2$ such that
\begin{equation}\label{eq:trace}
 \PP(h_t\in dy,\ (D_t,W)\in dD\,dW)=dy\,\nu_{t,y}(dD,dW).
\end{equation}
They are weakly continuous in $(t,y)\in\Z_2\times\R$, and for every
finite $m\geq0$,
\begin{equation}\label{eq:tracemoments}
 \sup_{t,y}\int\omega(D,W)^m\,\nu_{t,y}(dD,dW)<\infty.
\end{equation}
For each admissible $\psi$, the density
$f_\psi(t,y):=\int\psi\,d\nu_{t,y}$ is continuous.
\end{corollary}

\begin{proof}
Write $Z_t:=(D_t,W)$. For admissible $\psi$, the finite signed measure
\[
 \lambda_{t,\psi}(B):=\E_V\left[\psi(Z_t)
                         \ind_{\{h_t\in B\}}\right]
\]
has integrable Fourier transform by \eqref{eq:weightedchar}. Fourier
inversion \cite[Theorem~3.3.14]{Durrett} therefore gives the continuous
density
\begin{equation}\label{eq:inverse}
 f_\psi(t,y)=\frac1{2\pi}\int_\R e^{-i\xi y}
       \E_V[\psi(Z_t)e^{i\xi h_t}]\,d\xi.
\end{equation}
Lemma~\ref{lem:moments}, \eqref{eq:continuityt}, and truncation in
$\omega$ show that the integrand is continuous in $t$ for each $\xi$;
the majorant \eqref{eq:weightedchar} makes $f_\psi$ jointly continuous
in $(t,y)$.

If $\psi\geq0$, then $\lambda_{t,\psi}$ is positive, so its continuous
density is nonnegative. Thus, for fixed $(t,y)$, the map
$\psi\mapsto f_\psi(t,y)$ is a positive functional on
$C_c^\infty((0,\infty)^2)$, bounded by
$\|\psi\|_\infty f_1(t,y)$. It extends to $C_c((0,\infty)^2)$, and
Riesz representation gives a finite positive measure $\nu_{t,y}$.

Choose smooth cutoffs $0\leq\chi_R\leq1$ with $\chi_R\uparrow1$,
equal to one on $\{\omega\leq R\}$ and supported on
$\{\omega\leq2R\}$. Positivity
of weighted densities gives, for every $m>0$,
\begin{equation}\label{eq:masscontrol}
 0\leq f_1(t,y)-f_{\chi_R}(t,y)
 =f_{1-\chi_R}(t,y)
 \leq R^{-m}f_{\omega^m}(t,y)\ll_mR^{-m}.
\end{equation}
Hence $\nu_{t,y}((0,\infty)^2)=f_1(t,y)$. More generally, for $r>0$,
\[
 0\leq f_{\omega^m}(t,y)-f_{\omega^m\chi_R}(t,y)
 \leq R^{-r}f_{\omega^{m+r}}(t,y).
\]
Letting $R\to\infty$ gives
\[
 \int\omega^m\,d\nu_{t,y}=f_{\omega^m}(t,y),
 \qquad \sup_{t,y}f_{\omega^m}(t,y)<\infty,
\]
which proves \eqref{eq:tracemoments}; the same cutoff argument extends
$\int\psi\,d\nu_{t,y}=f_\psi(t,y)$ to every admissible $\psi$.
Continuity on compactly supported tests, together with
\eqref{eq:tracemoments}, gives weak continuity of
$(t,y)\mapsto\nu_{t,y}$. Finally, for $g\in C_c(\R)$ and
$\psi\in C_c((0,\infty)^2)$,
\[
 \int_\R g(y)\!\int\psi\,d\nu_{t,y}\,dy
   =\E_V[g(h_t)\psi(Z_t)],
\]
and a monotone-class argument gives \eqref{eq:trace}.
\end{proof}

The measure relevant to a small gap is therefore
\begin{equation}\label{eq:nustar}
 \nu_*:=\nu_{0,0}=\lim_{(t,y)\to(0,0)}\nu_{t,y}.
\end{equation}
It is finite and has moments of every order appearing below.
Section~\ref{sec:positivity} proves that its total mass is positive.

We can now state the coefficient in the small-gap asymptotic. Put
$\pi_1:=1/5$, $\pi_{3/2}:=4/5$. For $\alpha\in\R$, set
$\eta_j:=2^{\alpha-j}$ and define
\begin{equation}\label{eq:profile}
 \Phi(\alpha):=\frac32\sum_{j\in\Z}\eta_j^2
 \sum_{b\in\{1,3/2\}}\pi_b
 \int\bigl(D^2-4DbW\eta_j\bigr)_+^{1/2}\,\nu_*(dD,dW).
\end{equation}
We prove
\begin{equation}\label{eq:asymptotic}
 F_{\rm raw}(s)=s^5\Phi(\log_2(s^{-2}))+o(s^5).
\end{equation}
The coefficient in Theorem~\ref{thm:spacing} is then
\begin{equation}\label{eq:profileconversion}
 \Psi(\alpha):=\rho^{-5}\Phi(\alpha+2\log_2\rho).
\end{equation}
\begin{proof}[Proof of Theorem~\ref{thm:spacing}]
Let
\[
 Q(s):=\PP(q_t\leq s\mid\beta_2(t)>0),\qquad
 \log_2(s^{-2})=L+\alpha,
\]
where $L\in\Z$ and $0\leq\alpha<1$, and put
$\eta_j:=2^{\alpha-j}$. For $\ell\geq0$ and
$b\in\{1,3/2\}$ write
\[
 \overline\nu_{\ell,b,y}:=
   \E_{t\mid\ell,b}[\nu_{t,y}].
\]
Since $2^{-\ell}=s^2\eta_j$ and
$4^{-\ell}=s^4\eta_j^2$ when $\ell=L+j$, equations
\eqref{eq:marks}, \eqref{eq:residue}, \eqref{eq:quadratic-event}, and
\eqref{eq:trace} give the single identity
\begin{equation}\label{eq:classintegral}
 \frac{Q(s)}{s^5}
 =\frac34\sum_{b\in\{1,3/2\}}\pi_b
   \sum_{\substack{j\in\Z\\j\geq-L}}\eta_j^2
   \int_\R\!\int
   \ind_{\{z^2\leq D^2-4DbW\eta_j\}}
   \,\overline\nu_{L+j,b,sz}(dD,dW)\,dz .
\end{equation}
The indicator forces $D\geq|z|$ and $D/W\geq4b\eta_j$. Since
$(D/W)^3\max(1,D^2)\ll\omega(D,W)^8$,
\eqref{eq:tracemoments} gives, uniformly in every parameter,
\begin{equation}\label{eq:classtail}
 \int
   \ind_{\{z^2\leq D^2-4DbW\eta_j\}}
   \,\overline\nu_{L+j,b,sz}(dD,dW)
 \ll (1+z^2)^{-1}\min(1,\eta_j^{-3}).
\end{equation}
Consequently
\[
 \sum_{j\in\Z}\eta_j^2\min(1,\eta_j^{-3})\ll1
 \qquad(0\leq\alpha<1),
\]
so \eqref{eq:classintegral} has an integrable majorant on
$\Z\times\R$ and, in particular, $Q(s)\ll s^5$.

Now let $s_n\downarrow0$ with $\alpha(s_n)\to\alpha_0$. For fixed
$j$, $L_n+j\to\infty$, and every $t$ in the corresponding class lies
in $4^{L_n+j}\Z_2$. Corollary~\ref{prop:density} therefore gives
\[
 \overline\nu_{L_n+j,b,s_nz}\Longrightarrow\nu_*,
 \qquad
 \eta_j(s_n)\longrightarrow\eta_j^0:=2^{\alpha_0-j}.
\]
For each $(D,W)$, the boundary equation
$z^2=D^2-4DbW\eta_j^0$ has at most two solutions in $z$; Fubini shows
that the limiting boundary has $\nu_*$-mass zero for almost every $z$.
After restricting to $|z|\leq R$ and $\omega(D,W)\leq R$, the defining
inequalities converge uniformly. Weak convergence and the boundary
statement therefore give convergence of the corresponding integrals;
\eqref{eq:classtail} and \eqref{eq:tracemoments} remove the truncation
uniformly. Dominated convergence in \eqref{eq:classintegral} now gives
\begin{align*}
 \lim_{n\to\infty}\frac{Q(s_n)}{s_n^5}
 &=\frac34\sum_{j\in\Z}(\eta_j^0)^2
   \sum_b\pi_b\int_\R\!\int
   \ind_{\{z^2\leq D^2-4DbW\eta_j^0\}}\,d\nu_*\,dz\\
 &=\frac32\sum_{j\in\Z}(\eta_j^0)^2
   \sum_b\pi_b\int
   (D^2-4DbW\eta_j^0)_+^{1/2}\,d\nu_*
 =\Phi(\alpha_0).
\end{align*}
The same majorant proves continuity of $\Phi$, while
$\eta_j(\alpha+1)=\eta_{j-1}(\alpha)$ proves one-periodicity.
Compactness of $[0,1]$ and the preceding subsequence argument give
\[
 Q(s)=s^5\Phi(\log_2(s^{-2}))+o(s^5).
\]

Lemma~\ref{lem:quadratic} gives
\[
 Q(s-Cs^2)\leq F_{\rm raw}(s)\leq Q(s+Cs^2).
\]
Continuity and periodicity of $\Phi$ therefore prove
\eqref{eq:asymptotic}. Proposition~\ref{thm:positive} gives
$\nu_*\ne0$; on a compact box of positive $\nu_*$-mass the integrand
in \eqref{eq:profile} is positive for all sufficiently small
$\eta_j$. Thus $\Phi>0$, and continuity and periodicity give
$\min\Phi>0$. Unfolding gives \eqref{eq:profileconversion} and
\eqref{eq:unfoldedasymptotic}.
Finally Theorem~\ref{conv:thm:main} and Portmanteau
\cite[Theorem~2.1]{Billingsley} give, for each fixed $s>0$,
\[
 \limsup_{N\to\infty}\frac1N\#\{j\leq N:s_j(c)\leq s\}
 \leq\mu([0,s])\ll s^5.
\]
\end{proof}

\subsection{Uniform estimates for the local variables}\label{sec:local}

It remains to prove Proposition~\ref{thm:smoothing}. This subsection
establishes two preliminary facts. The local products and the Green
coefficients have moments of every fixed order, uniformly in the
$2$-adic parameter $t$. They also vary continuously with $t$. We begin
by recording the Fourier information on the local densities that will
later produce cancellation.

The Gauss-sum computation used in
\eqref{conv:eq:fourier} gives, for $p$ odd, $\nu\geq1$, and
$(a,p)=1$,
\begin{equation}\label{eq:fourier}
 |\widehat\beta_p(a/p^\nu)|=p^{-3\nu/2}.
\end{equation}
The constant Fourier coefficient is one. For
$\delta_p:=\beta_p-1$ this yields
\begin{equation}\label{eq:vp}
 \E\delta_p=0,\quad |\delta_p|\leq p^{-1},\quad
 \upsilon_p:=\E\delta_p^2=\sum_{\nu\geq1}\varphi(p^\nu)p^{-3\nu}
       =\frac1{p(p+1)}.
\end{equation}
If $p\nmid k-l$, the covariance is exactly
\begin{equation}\label{eq:cov}
 \E[\delta_p(V_p+k)\delta_p(V_p+l)]=-p^{-3}.
\end{equation}
Indeed, if
\[
 c_q(r):=\sum_{\substack{a\bmod q\\(a,q)=1}}e(ar/q),
\]
then Parseval expresses the covariance as
$\sum_{\nu\geq1}p^{-3\nu}c_{p^\nu}(k-l)$. When $p\nmid k-l$,
$c_p(k-l)=-1$ and the terms with $\nu\geq2$ vanish.

At the prime two, write
\[
 \beta_2(t)=\sum_{\alpha\in\Z[1/2]/\Z}c_2(\alpha)e_\alpha(t).
\]
For nonzero $\alpha$, let $q(\alpha)$ denote its denominator in
lowest terms. The constant coefficient is one, the coefficients with
$q(\alpha)=2$ vanish, and every coefficient with
$q(\alpha)=2^\nu$, $\nu\geq2$, has modulus
$2^{3/2}2^{-3\nu/2}$. Consequently
\begin{equation}\label{eq:dyadic}
 \sum_\alpha|c_2(\alpha)|=2+\sqrt2,\qquad
 \sum_{q(\alpha)>2^{2\ell}}|c_2(\alpha)|\ll2^{-\ell}.
\end{equation}
This series converges uniformly, including at zero.

For a function $C$ on $\Omega_o$ put
\[
 \mathcal H_t C:=\PV\sum_{k\ne0}\frac{\beta_2(t+k)C(V+k)}k,
 \qquad
 \mathcal D_t C:=\sum_{k\ne0}\frac{\beta_2(t+k)C(V+k)}{k^2}.
\]
\begin{lemma}\label{lem:Lq}
For $1<q<\infty$,
\begin{equation}\label{eq:Lq}
 \|\mathcal H_t C\|_q+\|\mathcal D_t C\|_q\leq C_q\|C\|_q,
\end{equation}
uniformly in $t$. For $\ell\geq0$, if $t-t'\in4^\ell\Z_2$, then
\[
 \|\mathcal H_t-\mathcal H_{t'}\|_{q\to q}
 +\|\mathcal D_t-\mathcal D_{t'}\|_{q\to q}\ll_q2^{-\ell}.
\]
\end{lemma}

\begin{proof}
The discrete Hilbert transform is bounded on $\ell^q(\Z)$
\cite[Theorem~1.1]{BK}. Modulation and Poisson averaging give the
same uniform bound for the kernels $r^{|k|}e(k\alpha)/k$, $0<r<1$.
Calder\'on's transference principle \cite{Calderon}, applied to the
shift $V\mapsto V+1$, therefore gives
\[
 \left\|\sum_{k\ne0}
   \frac{r^{|k|}e(k\alpha)C(V+k)}k\right\|_q
 \leq C_q\|C\|_q.
\]
As $r\uparrow1$, density of character polynomials gives a strong
$L^q$ limit. In $L^2$, the uniformly bounded symmetric partial-sum
multipliers identify this limit with the stated principal value.

Expanding the absolutely convergent Fourier series of $\beta_2$ gives
\[
 \|\mathcal H_tC\|_q
 \leq C_q\sum_\alpha|c_2(\alpha)|\,\|C\|_q,
 \qquad
 \|\mathcal D_tC\|_q
 \leq\tfrac32\sum_{k\ne0}k^{-2}\|C\|_q.
\]
If $t-t'\in4^\ell\Z_2$, the frequencies with
$q(\alpha)\leq2^{2\ell}$ cancel. Applying \eqref{eq:dyadic} to the
remaining frequencies proves both operator-continuity estimates.
\end{proof}

\begin{lemma}\label{lem:moments}
For any finite set $\cP$ of odd primes put
\[
 B(V):=\prod_{p\text{ odd},\ p\notin\cP}\beta_p(V_p),\qquad
 \overline W:=B(V),\qquad \overline D:=\mathcal D_t B.
\]
For every fixed $q>0$,
\[
 \sup_{\cP,k,t}\E\left[
 B(V+k)^q+B(V+k)^{-q}+\overline D^q+\overline D^{-q}
 +|\mathcal H_tB|^q\right]<\infty.
\]
In particular, $W,D_t$ have uniformly bounded positive and negative
moments, and $h_t$ has uniformly bounded absolute moments.
For $q>1$ and $t-t'\in4^\ell\Z_2$,
\begin{equation}\label{eq:continuityt}
 \|h_t-h_{t'}\|_q+\|D_t-D_{t'}\|_q\ll_q2^{-\ell}.
\end{equation}
\end{lemma}

\begin{proof}
Taylor expansion gives $\E\beta_p^{\pm q}=1+O_q(p^{-2})$.
Independence and the almost-sure positivity proved above imply
\begin{equation}\label{eq:productmoments}
 \sup_{\cP,k}\E B(V+k)^{\pm q}
 \leq\prod_{p>2}(1+C_qp^{-2})<\infty.
\end{equation}
Higher moments justify passage to the infinite products.
Minkowski bounds $\overline D$ by the summable weights $k^{-2}$.
For inverse moments choose
\begin{equation}\label{eq:k0}
 k_0:=\begin{cases}1,&t\equiv0,1\pmod4,\\-1,&t\equiv2,3\pmod4.
 \end{cases}
\end{equation}
Then $\beta_2(t+k_0)=3/2$ and
$\overline D\geq\tfrac32B(V+k_0)$, so \eqref{eq:productmoments}
also bounds $\E\overline D^{-q}$. Lemma~\ref{lem:Lq} bounds
$\mathcal H_tB$. Taking $\cP=\varnothing$ proves the assertions
for $W,D_t,h_t$, including \eqref{eq:continuityt}; lower moments
follow from H\"older.
\end{proof}

The lemma justifies the principal value in \eqref{eq:hD} and the
continuity in \(t\) used above. Finally,
\[
 \E\sum_{k\ne0}\frac{A_t(k)}{k^2}<\infty,
\]
so the \(O(k^{-2})\) regularized kernel gives the normal convergence
asserted in \eqref{eq:H}.

\subsection{Independent prime blocks}\label{sec:prime}

We now expose a block of large odd primes. After conditioning on all
coordinates outside the block, the contribution of each block prime is
independent. Each contribution is small, but the product of their
characteristic functions gives strong Fourier decay. The point of this
subsection is to show that the interactions involving two or more block
primes are small enough to retain rather than discard.

Let $Y$ be large, set
\[
 \cP:=\{p:Y<p\leq2Y\},\qquad N:=|\cP|\asymp Y/\log Y,
\]
and let $\cF$ be generated by the odd-prime coordinates outside
$\cP$. Remove the block from the odd-prime product by writing
\[
 B(V):=\prod_{\substack{p>2\\p\notin\cP}}\beta_p(V_p),
 \qquad B_k:=B(V+k).
\]
Then
\[
 W(V+k)=B_k\prod_{p\in\cP}
               \bigl(1+\delta_p(V_p+k)\bigr).
\]
Put
\begin{equation}\label{eq:scales}
 b_Y:=\sum_{p\in\cP}p^{-1}\ll(\log Y)^{-1},\qquad
 \sigma^2:=\sum_{p\in\cP}p^{-2}\asymp(Y\log Y)^{-1}.
\end{equation}
For $S\subseteq\cP$, put
\[
 W_S:=B(V)\prod_{p\in S}\delta_p(V_p),\qquad
 h_S:=\mathcal H_tW_S,\qquad D_S:=\mathcal D_tW_S.
\]
Thus $S$ records exactly which block primes occur in a term. The
product expansion is exact and gives
\begin{equation}\label{eq:block}
 h_t=h_\varnothing+\sum_{p\in\cP}h_p+\cR,
 \qquad \cR:=\sum_{|S|\geq2}h_S.
\end{equation}
Conditional on $\cF$, the one-prime terms
$h_p=h_{\{p\}}$ are independent and centered. The variables appearing
in the weight $\psi$ have the simultaneous expansion
\[
 Z:=(D_t,W),\qquad \overline Z:=(\mathcal D_tB,B),
\]
and
\begin{equation}\label{eq:Delta}
 \Delta Z:=Z-\overline Z
   =\sum_{S\ne\varnothing}(D_S,W_S).
\end{equation}
Conditional on $\cF$, each $S$-term uses only the coordinates in $S$.

\begin{lemma}
\label{lem:blockmoments}
For every fixed $q>1$,
\begin{equation}\label{eq:Rq}
 \|\cR\|_q\ll_q\sigma^2,
 \qquad
 \|\Delta Z\|_q\ll_q\sigma.
\end{equation}
Moreover, for every $S\subseteq\cP$,
\begin{equation}\label{eq:subsetnorm}
 \|h_S\|_q+\|D_S\|_q+\|W_S\|_q
 \ll_q\prod_{p\in S}p^{-1}.
\end{equation}
The coordinates of $Z$ and $\overline Z$ are pointwise comparable:
\begin{equation}\label{eq:compareZ}
 e^{-2b_Y}\leq\frac{D_t}{\overline D}\leq e^{b_Y},
 \qquad
 e^{-2b_Y}\leq\frac{W}{\overline W}\leq e^{b_Y}.
\end{equation}
\end{lemma}

\begin{proof}
Put $P:=\prod_{p\in\cP}(1+\delta_p)$ and $S_1:=\sum_{p\in\cP}\delta_p$.
Since $\sum|\delta_p|\leq b_Y=o(1)$ and $\sum\delta_p^2\leq\sigma^2$,
\begin{equation}\label{eq:productremainder}
 |P-1-S_1|\ll S_1^2+\sigma^2.
\end{equation}
In the expansion of $\E S_1^{2r}$, centering removes singleton
indices, while for every $j\geq2$,
\[
 \sum_{p\in\cP}\E|\delta_p|^j
 \leq Y^{-(j-2)}\sigma^2\leq\sigma^j.
\]
Hence $\|S_1\|_q\ll_q\sigma$ and
$\|P-1-S_1\|_q\ll_q\sigma^2$. Independence of $B$ and the block,
together with Lemmas~\ref{lem:Lq} and \ref{lem:moments}, gives
\[
 \|\mathcal H_t(B(P-1-S_1))\|_q\ll_q\sigma^2,
 \qquad \|\mathcal D_t(B(P-1))\|_q+\|B(P-1)\|_q\ll_q\sigma.
\]
The same operator bounds applied to $B\prod_{p\in S}\delta_p$
prove \eqref{eq:subsetnorm}. Finally
$e^{-2b_Y}\leq P(V+k)\leq e^{b_Y}$ and positivity of the
slope weights imply \eqref{eq:compareZ}.
\end{proof}

To prove contraction, we truncate $h_p$ to nearby poles. The following
Fourier estimate controls the omitted poles. For an integer $K\geq1$, let
\[
 L_K(\theta):=\sum_{0<|k|\leq K}\frac{e(k\theta)}k,
 \qquad
 T_K(\theta):=\PV\sum_{|k|>K}\frac{e(k\theta)}k.
\]
For every $p\geq2$, $K\geq1$, and real $\theta$, we have
\begin{equation}\label{eq:grid}
 \frac1p\sum_{a=0}^{p-1}|T_K(\theta+a/p)|^2
          \ll K^{-1}+p^{-1}.
\end{equation}
Indeed, the uniformly bounded sine harmonic sums and summation by parts give
\[
 |T_K(\theta)|\ll
 \min\left(1,\frac1{K\|\theta\|_{\R/\Z}}\right)
 \quad(\theta\notin\Z);
\]
at integers the tail is zero. On a shifted $p$-grid, at most a fixed
number of points lie at each successive distance $j/p$ from an integer.
Hence the left side of \eqref{eq:grid} is bounded by
\[
 \frac Cp\left(1+\sum_{j=1}^{p}
                  \min\left(1,\frac{p^2}{K^2j^2}\right)\right)
 \ll p^{-1}+K^{-1}.
\]

Take $K=\lfloor Y/4\rfloor$, so $2K<p$ throughout the block. Set
\begin{equation}\label{eq:truncated}
 w_k:=\frac{\beta_2(t+k)B_k}{k},\quad
 M_K:=\sum_{0<|k|\leq K}|w_k|,\quad
 X_p:=\sum_{0<|k|\leq K}w_k\delta_p(V_p+k),\quad r_p:=h_p-X_p.
\end{equation}
\begin{lemma}\label{lem:tail}
Uniformly in $t$ and $p\in\cP$,
\begin{equation}\label{eq:tail}
 \E_{\cF}\|r_p\|_{L^2(V_p)}^2\ll p^{-2}Y^{-1}.
\end{equation}
\end{lemma}

\begin{proof}
Fix first a dyadic frequency $\alpha$ of $\beta_2(t+k)$.
Write $\widehat B(\theta)$ for the Fourier coefficients of $B$
on the outside coordinates. Since $B$ has no $p$-coordinate,
the pairs $(\theta,a/p^\nu)$ are distinct characters. Parseval gives
\begin{align*}
 &\left\|\PV\sum_{|k|>K}
       \frac{e(k\alpha)B(V+k)\delta_p(V_p+k)}k\right\|_2^2\\
 &\quad=\sum_\theta|\widehat B(\theta)|^2
       \sum_{\nu\geq1}\sum_{a\bmod p^\nu}^{*}
                p^{-3\nu}|T_K(\alpha+\theta+a/p^\nu)|^2.
\end{align*}
The terms with denominator \(p\) contribute
$O(p^{-2}(K^{-1}+p^{-1}))$ by \eqref{eq:grid}; higher
\(p\)-power denominators contribute $O(p^{-4})$ since \(T_K\) is bounded.
With $K\asymp Y$ and $p\asymp Y$, their sum is
$O(p^{-2}Y^{-1})$. Finally $\|B\|_2\ll1$, and Minkowski sums
the dyadic frequencies at cost $2+\sqrt2$.
\end{proof}

The truncated term has a variance bounded below on an event determined
by the coordinates outside the block. Indeed, $2K<p$, so distinct
indices in the truncation are distinct modulo $p$.
The covariance formula \eqref{eq:cov} gives
\begin{equation}\label{eq:varX}
 \Var(X_p\mid\cF)
 =(\upsilon_p+p^{-3})\sum w_k^2-p^{-3}\left(\sum w_k\right)^2.
\end{equation}
Moreover $|X_p|\leq M_K/p$, and at the index in \eqref{eq:k0},
$|w_{k_0}|=(3/2)B_{k_0}$. Consequently, if
\begin{equation}\label{eq:weightgood}
 |w_{k_0}|\geq a,\quad M_K\leq M,\quad p\geq CM^2/a^2,
\end{equation}
then $\Var(X_p\mid\cF)\geq c_0a^2/p^2$. Also, for each fixed $q$,
\begin{equation}\label{eq:weightmoments}
 \E|w_{k_0}|^{-q}\ll_q1,\qquad
 \E M_K^q\ll_q(\log(2K))^q.
\end{equation}

The next estimate shows that this variance still gives contraction
after the truncation error is restored.

\begin{lemma}\label{lem:contraction}
Suppose centered real $X,r$ satisfy
\[
 |X|\leq M/p,\quad \Var X\geq c_0a^2/p^2,\quad
 \E r^2\leq\kappa a^2/p^2.
\]
If $TM/p\leq1/8$ and $\kappa$ is sufficiently small in terms of $c_0$,
then
\begin{equation}\label{eq:contraction}
 |\E e^{iT(X+r)}|\leq\exp(-c_1T^2a^2/p^2).
\end{equation}
Independence between $X$ and $r$ is not required.
\end{lemma}

\begin{proof}
Let $(X',r')$ be an independent copy. Since
$|T(X-X')|\leq1/4$, symmetrization gives
\[
 \|e^{iTX}-e^{iTX'}\|_2^2
 =2\E[1-\cos(T(X-X'))]\geq cT^2\Var X.
\]
The reverse triangle inequality and $|e^{iu}-e^{iv}|\leq|u-v|$ imply
\begin{equation}\label{eq:symmetrize}
 \sqrt{2(1-|\E e^{iT(X+r)}|^2)}
 \geq cT\sqrt{\Var X}-2T\|r\|_2
 \geq c'Ta/p
\end{equation}
for sufficiently small $\kappa$. Squaring and using $1-u\leq e^{-u}$
proves \eqref{eq:contraction}.
\end{proof}

\subsection{Proof of the weighted smoothing estimate}\label{sec:retain}

We now combine the preceding estimates. Write
$h_t=\ell+\cR$, where
$\ell:=h_\varnothing+\sum_{p\in\cP}h_p$ contains the independent
one-prime terms and $\cR$ contains all interactions. On a suitable
outside event, most of the one-prime terms contract the characteristic
function. A term supported on $S\subseteq\cP$ still leaves all primes
outside $S$ available for this contraction. We therefore retain the
linear contribution of every interaction term and estimate only the
quadratic Taylor remainder in absolute value.

\begin{lemma}\label{lem:block-smoothing}
Let $\mathsf P$ be a probability law. Conditionally on a sigma-field
$\cF$, let $(V_p)_{p\in\cP}$ be independent, put $N:=|\cP|$, and write
\[
 \ell=\ell_\varnothing+\sum_{p\in\cP}\ell_p,
\]
where $\ell_\varnothing$ and the $\ell_p$ are real-valued,
$\ell_\varnothing$ is $\cF$-measurable, and $\ell_p$ is measurable
with respect to $\cF\vee\sigma(V_p)$. Suppose that on an event
$G\in\cF$ there is an $\cF$-measurable set
$\cP_{\rm good}\subseteq\cP$, of cardinality at least $N/2$, such that
\[
 \left|\E_{\mathsf P}(e^{iT\ell_p}\mid\cF)\right|\leq e^{-\gamma}
 \qquad(p\in\cP_{\rm good}).
\]
If $Q_S\in L^1(\mathsf P)$ is measurable with respect to
$\cF\vee\sigma(V_p:p\in S)$ for every $S\subseteq\cP$, then
\begin{equation}\label{eq:block-smoothing}
 \left|\E_{\mathsf P}\left[
   \ind_Ge^{iT\ell}\sum_{S\subseteq\cP}Q_S\right]\right|
 \leq e^{-\gamma N/4}\sum_{|S|\leq N/4}\E_{\mathsf P}|Q_S|
      +\sum_{|S|>N/4}\E_{\mathsf P}|Q_S|.
\end{equation}
In particular, if $\cP\subset(Y,2Y]$, $N\asymp Y/\log Y$,
$\gamma\geq cT^2a^2/Y^2$, and
\[
 \E_{\mathsf P}|Q_S|\leq C(1+|T|)
                   \prod_{p\in S}\frac{C_0}{p},
\]
then the left side of \eqref{eq:block-smoothing} is
\begin{equation}\label{eq:block-smoothing-corollary}
 \ll C(1+|T|)\left\{
 \exp\left(-c'\frac{T^2a^2}{Y\log Y}\right)+e^{-c'Y}\right\}.
\end{equation}
\end{lemma}

\begin{proof}
For $|S|\leq N/4$, at least $N/4$ good coordinates lie outside $S$.
Condition first on $\cF$ and on the coordinates in $S$, and use the
remaining independent factors. This gives
\[
 \left|\E_{\mathsf P}[\ind_GQ_Se^{iT\ell}]\right|
 \leq e^{-\gamma N/4}\E_{\mathsf P}|Q_S|.
\]
For larger $S$, use the trivial bound and sum. Finally, with
$b:=\sum_{p\in\cP}C_0/p\ll1/\log Y$,
\[
 \sum_S\prod_{p\in S}\frac{C_0}{p}\leq e^b\ll1,
 \qquad
 \sum_{|S|>N/4}\prod_{p\in S}\frac{C_0}{p}
 \leq\sum_{j>N/4}\frac{b^j}{j!}\leq e^{-c'Y}.
\]
This proves both assertions.
\end{proof}

Set $T:=|\xi|$ large. We choose
\begin{equation}\label{eq:parameters}
 Y:=T^{7/4},\quad K:=\lfloor Y/4\rfloor,\quad
 a:=T^{-1/16},\quad M:=T^{1/16}\log(2K).
\end{equation}
The block scale $Y$ makes the interaction remainder small, while the
thresholds $a$ and $M$ retain enough variance in each one-prime term.
Let $G$ be the $\cF$-measurable event on which
$|w_{k_0}|\geq a$, $M_K\leq M$, and
\begin{equation}\label{eq:tailgood}
 \frac1N\sum_{p\in\cP}p^2\|r_p\|_{L^2(V_p)}^2
                  \leq\kappa a^2/2.
\end{equation}
By Lemma~\ref{lem:tail}, \eqref{eq:weightmoments}, and Markov,
\begin{equation}\label{eq:bad}
 \PP(G^c)\ll T^{-26/16}+a^{-2}Y^{-1}\ll T^{-13/8},
\end{equation}
using moments of order $26$.
On $G$, at least $N/2$ primes satisfy
$\|r_p\|_2^2\leq\kappa a^2/p^2$. The choice \eqref{eq:parameters} ensures
$p\geq CM^2/a^2$ and $TM/p\leq1/8$ for large $T$, so
Lemma~\ref{lem:contraction} applies to these primes.

\begin{proof}[Proof of Proposition~\ref{thm:smoothing}]
It is enough to consider $T:=|\xi|$ large, since bounded $T$ follows
from Lemma~\ref{lem:moments}. Use the block and the good event above,
and write $h_t=\ell+\cR$ as in \eqref{eq:block}.

For admissible $\psi$, Taylor expansion and \eqref{eq:compareZ} give
\begin{equation}\label{eq:retain}
 \psi(Z)e^{iT\cR}
 =\psi(\overline Z)+\nabla\psi(\overline Z)\cdot\Delta Z
       +iT\psi(\overline Z)\cR+\mathcal E,
\end{equation}
where, for some fixed $m$,
\begin{equation}\label{eq:Taylorerror}
 |\mathcal E|\leq C_\psi\omega(\overline Z)^m
 \bigl(|\Delta Z|^2+T|\Delta Z|\,|\cR|+T^2|\cR|^2\bigr).
\end{equation}
By H\"older, Lemma~\ref{lem:moments}, and \eqref{eq:Rq},
\begin{equation}\label{eq:erroraverage}
 \E|\mathcal E|\ll_\psi\sigma^2+T\sigma^3+T^2\sigma^4
                 \ll_\psi\sigma^2+T^2\sigma^4.
\end{equation}

The retained terms have the supported form required by
Lemma~\ref{lem:block-smoothing}. Namely, put
\[
 Q_\varnothing:=\psi(\overline Z)
\]
and, for $S\ne\varnothing$,
\[
 Q_S:=\nabla\psi(\overline Z)\cdot(D_S,W_S)
      +\ind_{\{|S|\geq2\}}iT\psi(\overline Z)h_S.
\]
By H\"older, Lemmas~\ref{lem:moments} and \ref{lem:blockmoments}, and
the polynomial growth of $\psi$ and its derivatives,
\[
 \E|Q_S|\ll_\psi(1+T)\prod_{p\in S}\frac{C_0}{p}
\]
for a fixed absolute $C_0$. The contraction estimate above verifies
the good-coordinate hypothesis of Lemma~\ref{lem:block-smoothing}.
Since $\|\psi(Z)\|_{13}\ll_\psi1$, that lemma, \eqref{eq:erroraverage},
and H\"older on $G^c$ give
\[
 |\E[\psi(Z)e^{iTh_t}]|
 \ll_\psi(1+T)
       \left(e^{-cT^2a^2/(Y\log Y)}+e^{-cY}\right)
       +\sigma^2+T^2\sigma^4+\PP(G^c)^{12/13}
 \ll_\psi T^{-3/2}.
\]
Indeed, the exponential terms are negligible,
$\sigma^2\ll T^{-7/4}$, $T^2\sigma^4\ll T^{-3/2}$, and
\eqref{eq:bad} gives $\PP(G^c)^{12/13}\ll T^{-3/2}$.
This proves \eqref{eq:weightedchar}; the argument is identical for
either sign of $\xi$.
\end{proof}

\section{Positivity of the density}\label{sec:positivity}

\begin{proposition}\label{thm:positive}
The continuous density of $h_0$ is strictly positive at every real point:
\begin{equation}\label{eq:positivity}
 f_1(0,y)>0\qquad(y\in\R).
\end{equation}
In particular $\nu_{0,0}\ne0$.
\end{proposition}

Write $h:=h_0$ and $f(y):=f_1(0,y)$. We combine uniform smoothing
under finite changes of measure with a limit that moves one pole
contribution, then average the displacement. Uniform smoothing
allows positivity to pass through the limit.

\subsection{A class of tilted product measures}
All prime products in this section are over odd primes.
Fix a nonzero integer $k_*$, a number $0<\vartheta_0<1$, and a finite
set $I$ of odd primes. For $0\leq\vartheta_p\leq\vartheta_0$ define
\begin{equation}\label{eq:tilt}
 \frac{dQ}{d\PP}(V):=
 \prod_{p\in I}\bigl(1+\vartheta_p p\delta_p(V_p+k_*)\bigr),
 \qquad \sum_{p\in I}\frac{\vartheta_p}{p}\leq L.
\end{equation}
These are probability measures with independent coordinates. Each
local density lies between $1-\vartheta_0$ and $1+\vartheta_0$.
For a fixed finite $I$, the full density is bounded above and below
by positive constants; these bounds need not be uniform in $I$.
Use $\vartheta_p=0$ off $I$ and write $Q_p$ for the local law.
We write $\|X\|_{q,Q}:=(\E_Q|X|^q)^{1/q}$.

Let $\mathcal A$ be the algebra of integer sequences with absolutely
summable rational Fourier coefficients, with norm the sum of their
absolute values. The covariance formula gives
\begin{align}
 m_p(k)&:=\E_{Q_p}\delta_p(V_p+k)
       =\vartheta_p p C_p(k-k_*),\label{eq:tiltmean}\\
 C_p(a)&:=\sum_{\nu\geq1}p^{-3\nu}c_{p^\nu}(a),&
 \|m_p\|_{\mathcal A}
       &=\vartheta_p p\upsilon_p=\frac{\vartheta_p}{p+1}.
       \label{eq:tiltWiener}
\end{align}
Set $\mu_p(k):=1+m_p(k)$ and
$\xi_p(V_p,k):=\delta_p(V_p+k)-m_p(k)$.
For any set of indices $E$,
\begin{equation}\label{eq:meanproduct}
 \left\|\prod_{p\in E}\mu_p\right\|_{\mathcal A}\leq e^L.
\end{equation}
Thus the deterministic mean factors have a uniform Fourier norm
even when the full Radon--Nikodym derivative is very large.

\begin{lemma}\label{lem:modulated}
For $F\in\mathcal A$, define
\[
 \mathcal T_F C:=\PV\sum_{k\ne0}\frac{\beta_2(k)F(k)C(V+k)}k,
 \qquad
 \mathcal T_{F,>K} C:=\PV\sum_{|k|>K}
                         \frac{\beta_2(k)F(k)C(V+k)}k.
\]
Under Haar measure,
\begin{align}
 \|\mathcal T_FC\|_2&\leq C\|F\|_{\mathcal A}\|C\|_2,
                    \label{eq:modulatedfull}\\
 \|\mathcal T_{F,>K}(C\delta_p)\|_2^2
 &\leq C\|F\|_{\mathcal A}^2\|C\|_2^2
                  p^{-2}(K^{-1}+p^{-1}),\label{eq:modulatedtail}
\end{align}
where the second assertion assumes that $C$ is independent of $V_p$.
\end{lemma}
\begin{proof}
Expand $F$ in its absolutely summable Fourier series. Each frequency
modulates the operator in Lemma~\ref{lem:Lq} and shifts the argument
of $T_K$ in the Parseval calculation of Lemma~\ref{lem:tail}.
The grid bound \eqref{eq:grid} is uniform in this shift; the higher conductors
cost $O(p^{-4})$. Minkowski sums the frequencies at cost
$\|F\|_{\mathcal A}$ in the norm.
\end{proof}

\begin{lemma}\label{lem:tiltestimates}
Let $J:=\{p:Y<p\leq2Y\}$, and let $\cF$ be generated by the coordinates
outside $J$. Put $B(V+k):=\prod_{p\notin J}\beta_p(V_p+k)$.
The exact conditional centered expansion is
\begin{equation}\label{eq:tiltchaos}
 h=\sum_{S\subseteq J}h_S^Q,\qquad
 h_S^Q:=\PV\sum_{k\ne0}\frac{\beta_2(k)B(V+k)}k
       \prod_{p\in J\setminus S}\mu_p(k)
       \prod_{p\in S}\xi_p(V_p,k).
\end{equation}
There is an absolute constant $C_0\geq1$ such that, with constants
depending only on $L,\vartheta_0$,
\begin{align}
 \|h_S^Q\|_{2,Q}&\leq C_L\prod_{p\in S}\frac{C_0}{p},
       &\left\|\sum_{|S|\geq2}h_S^Q\right\|_{2,Q}
          &\ll_L\sum_{p\in J}p^{-2}.\label{eq:tiltchaosbounds}
\end{align}
For every fixed $q>0$, $\E_Q B(V+k)^{\pm q}\ll_{q,L}1$ uniformly
in $k,J,I$. If $X_p$ truncates $h_{\{p\}}^Q$ to $0<|k|\leq K$,
then
\begin{equation}\label{eq:tilttail}
 \E_Q|h_{\{p\}}^Q-X_p|^2
       \ll_L p^{-2}(K^{-1}+p^{-1}).
\end{equation}
The first estimate in \eqref{eq:tiltchaosbounds} also holds for any
finite $J$, with the same constants.
\end{lemma}

\begin{proof}
Expand the outside tilted primes as well. For $S\subseteq J$ and
$T\subseteq I\setminus J$, put
\[
 E:=S\cup T,\qquad B_0:=\prod_{p\notin I\cup J}\beta_p,
 \qquad F_{S,T}:=\prod_{p\in(I\cup J)\setminus E}\mu_p.
\]
With the commuting orthogonal projections $\Pi_p:=1-\E_{Q_p}$,
the term having centered support $E$ is
\begin{equation}\label{eq:projectionidentity}
 h_{S,T}^Q:=\left(\prod_{p\in E}\Pi_p\right)
       \mathcal T_{F_{S,T}}\left(B_0\prod_{p\in E}\delta_p\right).
\end{equation}
This follows first for finite sums and then in $L^2$.
The input depends only on $E$ and Haar coordinates, so comparison
with Haar costs only $(1+\vartheta_0)^{|E|}$.
Since $\|F_{S,T}\|_{\mathcal A}\leq e^L$, projection and
Lemma~\ref{lem:modulated} give
\begin{equation}\label{eq:orthotilt}
 \|h_{S,T}^Q\|_{2,Q}^2
 \leq Ce^{2L}\|B_0\|_2^2
                \prod_{p\in E}(1+\vartheta_0)\upsilon_p.
\end{equation}
Different centered supports are orthogonal. Summing over $T$ costs
the bounded product $\prod_{p\in I\setminus J}(1+(1+\vartheta_0)\upsilon_p)$,
and hence
\[
 \|h_S^Q\|_{2,Q}^2\leq C_L\prod_{p\in S}C_0^2p^{-2},\qquad
 \left\|\sum_{|S|\geq2}h_S^Q\right\|_{2,Q}^2
 \leq C_L\sum_{|S|\geq2}\prod_{p\in S}C_0^2p^{-2}
 \ll_L\left(\sum_{p\in J}p^{-2}\right)^2.
\]
This also justifies the expansions and applies to any finite $J$.

For the moments, Taylor expansion gives
\[
 \E_{Q_p}\beta_p(V_p+k)^{\pm q}
 =1\pm q m_p(k)+O_q(p^{-2}),\qquad \sum_p|m_p(k)|\leq L.
\]
Multiplication over independent coordinates proves the claimed bounds.
For the tail use \eqref{eq:projectionidentity} with $S=\{p\}$
and replace $\mathcal T_F$ by $\mathcal T_{F,>K}$.
Before projection the input $C=B_0\prod_{q\in T}\delta_q$
is independent of $V_p$, so \eqref{eq:modulatedtail} gives
\[
 \|\mathcal T_{F_{\{p\},T},>K}(C\delta_p)\|_{2,\PP}^2
 \ll e^{2L}\|B_0\|_2^2\prod_{q\in T}\upsilon_q\,
                      p^{-2}(K^{-1}+p^{-1}).
\]
Comparison on $T\cup\{p\}$, contraction of the projections, and
orthogonal summation over $T$ prove \eqref{eq:tilttail}.
\end{proof}

\begin{proposition}\label{prop:tiltsmoothing}
All the measures \eqref{eq:tilt} satisfy
\begin{equation}\label{eq:tiltchar}
 |\E_Q e^{i\xi h}|\ll_{L,\vartheta_0}(1+|\xi|)^{-3/2}.
\end{equation}
In particular their $h$-densities are continuous and have a common
integrable Fourier majorant. The bound does not involve $|I|$ or
the supremum of $dQ/d\PP$.
\end{proposition}

\begin{proof}
Put $T:=|\xi|$ and assume that $T$ is large; bounded $T$ is trivial.
Use the block $J=(Y,2Y]\cap\{\text{odd primes}\}$ and parameters
\eqref{eq:parameters}. Write $h=\ell_Q+\cR_Q$, with
$\ell_Q:=h_\varnothing^Q+\sum_{p\in J}h_{\{p\}}^Q$.
The one-prime truncations are
\[
 X_p:=\sum_{0<|k|\leq K}w_{p,k}\xi_p(V_p,k),\qquad
 w_{p,k}:=\frac{\beta_2(k)B(V+k)}k
                     \prod_{q\in J\setminus\{p\}}\mu_q(k).
\]
The mean products lie between $e^{-2b_Y}$ and $e^{b_Y}$.
On the outside event
\[
 \tfrac32B(V+1)\geq2a,\qquad
 \sum_{0<|k|\leq K}\frac{\beta_2(k)B(V+k)}{|k|}\leq M/4,
\]
we therefore have $|w_{p,1}|\geq a$ and $|X_p|\leq M/p$.
Since $dQ_p/d\PP_p\geq1-\vartheta_0$,
\[
 \Var_{Q_p}X\geq(1-\vartheta_0)\Var_{\PP_p}X.
\]
Apply this to $\sum w_{p,k}\delta_p(V_p+k)$ and use
\eqref{eq:varX}: $\Var_{Q_p}X_p\geq c_{\vartheta_0}a^2/p^2$.
Lemma~\ref{lem:tiltestimates} supplies the uniform moment and tail
bounds. The argument of \eqref{eq:bad} gives an outside event $G$ with
\begin{equation}\label{eq:tiltbad}
 Q(G^c)\ll_{L,\vartheta_0}T^{-13/8},
\end{equation}
on which at least half the block primes satisfy
Lemma~\ref{lem:contraction}.

Apply Lemma~\ref{lem:block-smoothing} under the law $\mathsf P=Q$,
with $\ell=\ell_Q$. The expansion
\[
 e^{iT\cR_Q}=1+iT\cR_Q+\mathcal E_Q,
 \qquad |\mathcal E_Q|\leq\tfrac12T^2|\cR_Q|^2,
\]
has the required supported form: take $Q_\varnothing=1$,
$Q_S=0$ for $|S|=1$, and $Q_S=iTh_S^Q$ for $|S|\geq2$.
By \eqref{eq:tiltchaosbounds}, uniformly in $I$,
\[
 \E_Q|Q_S|\ll_{L,\vartheta_0}(1+T)
       \prod_{p\in S}\frac{C_0}{p},
 \qquad
 \E_Q|\mathcal E_Q|\ll_{L,\vartheta_0}T^2\sigma^4.
\]
The contraction above verifies the good-coordinate hypothesis.
Lemma~\ref{lem:block-smoothing} and \eqref{eq:tiltbad} therefore give
\[
 |\E_Qe^{iTh}|
 \ll_{L,\vartheta_0}(1+T)
 \left(e^{-cT^2a^2/(Y\log Y)}+e^{-cY}\right)
       +T^2\sigma^4+Q(G^c)
 \ll_{L,\vartheta_0}T^{-3/2}.
\]
This proves \eqref{eq:tiltchar}; the estimate is uniform in $I$.
\end{proof}

\subsection{Changing one pole contribution in the limit}
Fix $A>1$ and $0\leq \vartheta\leq\vartheta_0$. Let $Q_{P,\vartheta}$ use
$I_P:=\{p:P<p\leq P^A\}$ and $\vartheta_p:=\vartheta$ in \eqref{eq:tilt}.
Mertens' prime sum formula gives
$\sum_{p\in I_P}p^{-1}\to\log A$.
Define the signed pole contribution
\begin{equation}\label{eq:astar}
 a_*(V):=\frac{\beta_2(k_*)}{k_*}W(V+k_*).
\end{equation}

\begin{proposition}\label{prop:displacement}
As $P\to\infty$,
\begin{equation}\label{eq:displacement}
 \Law_{Q_{P,\vartheta}}(h)\ \Longrightarrow\
 \Law_{\PP}\bigl(h+(A^\vartheta-1)a_*\bigr).
\end{equation}
Convergence of characteristic functions is uniform for
$\vartheta\in[0,\vartheta_0]$ at each fixed frequency.
\end{proposition}

\begin{proof}
Put $B_P(V):=\prod_{p\notin I_P}\beta_p(V_p)$ and
$\sigma_P^2:=\sum_{p\in I_P}p^{-2}=o(1)$.
The outside coordinates have Haar law. The conditional mean is
\[
 \overline h_{P,\vartheta}:=\PV\sum_{k\ne0}
       \frac{\beta_2(k)B_P(V+k)}k M_{P,\vartheta}(k),\qquad
 M_{P,\vartheta}(k):=\prod_{p\in I_P}(1+\vartheta p C_p(k-k_*)).
\]
The centered-support estimate and orthogonality in
Lemma~\ref{lem:tiltestimates}, with $J=I_P$, give
\begin{equation}\label{eq:tiltconcentration}
 \|h-\overline h_{P,\vartheta}\|_{2,Q_{P,\vartheta}}\ll_{A,\vartheta_0}\sigma_P.
\end{equation}
The spatial concentration of the mean follows from the Ramanujan-sum identity
\begin{equation}\label{eq:covdivisor}
 p C_p(a)=-p^{-2}+(1-p^{-3})
                 \sum_{\nu\geq1}p^{1-2\nu}\ind_{\{p^\nu\mid a\}}.
\end{equation}
Consequently
\begin{equation}\label{eq:meandivisor}
 M_{P,\vartheta}(k)=\kappa_{P,\vartheta}\sum_d a_\vartheta(d)\ind_{\{d\mid k-k_*\}},
 \qquad \kappa_{P,\vartheta}:=\prod_{p\in I_P}(1-\vartheta/p^2),
\end{equation}
where $a_\vartheta$ is multiplicative on $I_P$-smooth integers and
\[
 a_\vartheta(1):=1,\qquad
 a_\vartheta(p^\nu):=\frac{\vartheta(1-p^{-3})p^{1-2\nu}}{1-\vartheta/p^2}\geq0.
\]
Its total mass is uniformly bounded, because
\begin{equation}\label{eq:centralmultiplier}
 \kappa_{P,\vartheta}\sum_d a_\vartheta(d)=M_{P,\vartheta}(k_*)
       =\prod_{p\in I_P}(1+\vartheta/(p+1))\longrightarrow A^\vartheta
\end{equation}
uniformly in $\vartheta$. Also $\kappa_{P,\vartheta}\to1$ uniformly.

If $d>2|k_*|$, the restricted Hilbert sum satisfies
\begin{equation}\label{eq:progressionHilbert}
 \left\|\PV\sum_{\substack{k\ne0\\k\equiv k_*\ (d)}}
       \frac{\beta_2(k)B_P(V+k)}k
       -\frac{\beta_2(k_*)B_P(V+k_*)}{k_*}\right\|_2
       \ll d^{-1}.
\end{equation}

To see this, remove $k=k_*$ and write $k=k_*+dj$. The resulting
operator is $d^{-1}$ times a modulated Hilbert transform along
$V\mapsto V+d$, with kernel $(j+k_*/d)^{-1}$. Its difference from
$1/j$ has $\ell^1$ norm $O(|k_*|/d)$. The Fourier norm of
$\beta_2(k_*+dj)$ is at most $2+\sqrt2$, proving the bound.
The same argument is uniform in symmetric cutoffs $|k|\leq L$
with $L\geq|k_*|$, so $\sum_d a_\vartheta(d)\ll_A1$ permits
interchange of the divisor series and principal value in $L^2$.
Since every nontrivial divisor exceeds $P$, the total error is
$O(\sum_{d>1}a_\vartheta(d)/d)=O_A(P^{-1})$. Hence
\begin{equation}\label{eq:meanlimit}
 \overline h_{P,\vartheta}
 =\kappa_{P,\vartheta}\mathcal H_0 B_P+
   (M_{P,\vartheta}(k_*)-\kappa_{P,\vartheta})
      \frac{\beta_2(k_*)B_P(V+k_*)}{k_*}
   +O_{L^2}(P^{-1}).
\end{equation}
All constants may depend on fixed $A,k_*,\vartheta_0$ and are uniform in $\vartheta$.
Under Haar, $W=B_PP_{I_P}$, where
$P_{I_P}:=\prod_{p\in I_P}\beta_p(V_p)$, and the two factors are
independent. Hence
\[
 \|B_P-W\|_2^2
 =\E B_P^2\,\E(P_{I_P}-1)^2
 =\E B_P^2\left(\prod_{p\in I_P}(1+\upsilon_p)-1\right)
 \ll\sigma_P^2.
\]
The Hilbert bound and \eqref{eq:centralmultiplier} now give
\[
 \sup_{\vartheta\in[0,\vartheta_0]}
 \|\overline h_{P,\vartheta}-(h+(A^\vartheta-1)a_*)\|_{2,\PP}
 \longrightarrow0.
\]
The conditional mean depends only on outside coordinates, which
have Haar law under both measures. Combining this limit with
\eqref{eq:tiltconcentration} proves convergence of characteristic
functions, uniformly in $\vartheta$ at each fixed frequency.
\end{proof}

\subsection{Averaging displacements forces positivity}
\begin{proof}[Proof of Proposition~\ref{thm:positive}]
Fix $y\in\R$ and choose $\varepsilon\in\{1,-1\}$ with
$\PP(\varepsilon(y-h)>0)>0$, possible since $h$ has a density.
Use $k_*:=1$ for $\varepsilon=1$ and $k_*:=-2$ otherwise.
Then $a_*=\varepsilon a$, where $a:=|a_*|>0$ almost surely and
$\E a^{-1}<\infty$ by Lemma~\ref{lem:moments}.
Choose $M<\infty$ with $\PP(0<\varepsilon(y-h)/a<M)>0$ and
a smooth probability density $r$, supported on $[0,M]$ and
positive on $(0,M)$. For independent $U$ with density $r$,
$h+\varepsilon Ua$ has continuous density
\begin{equation}\label{eq:mixdensity}
 g(x):=\E\left[a^{-1}r\left(\varepsilon(x-h)/a\right)\right],
 \qquad g(y)>0.
\end{equation}
Continuity follows from $\E a^{-1}<\infty$.

Take $\vartheta_0:=1/2$, $A:=(1+M)^2$, and
$\vartheta(u):=\log(1+u)/\log A$. Let
$\overline Q_P:=\int_0^M r(u)Q_{P,\vartheta(u)}\,du$.
Its $h$-law converges weakly to $g(x)\,dx$ by
Proposition~\ref{prop:displacement}. Their budgets are uniformly
bounded, so Proposition~\ref{prop:tiltsmoothing} supplies a common
integrable Fourier majorant. Fourier inversion gives
\begin{equation}\label{eq:uniformdensitylimit}
 \|g_P-g\|_\infty\leq(2\pi)^{-1}
       \int_\R|\widehat g_P-\widehat g|\longrightarrow0.
\end{equation}
Fix a finite $P$ with $g_P(y)>0$. Set
$C_P:=(3/2)^{|I_P|}$. Since $d\overline Q_P/d\PP\leq C_P<\infty$,
we have $g_P\leq C_Pf$ almost everywhere, hence everywhere by
continuity. Therefore
\begin{equation}\label{eq:positivityconclusion}
 f(y)\geq C_P^{-1}g_P(y)>0.
\end{equation}
No uniform bound on $C_P$ is needed.
\end{proof}

\subsection*{Declaration on the use of artificial intelligence}
This paper was developed with assistance from OpenAI Codex in ChatGPT Work,
using models in the GPT-5 family, accessed during September--October 2026. The idea behind the proof was supplied by the author and
these models were used to identify and explore the problem, suggest
intermediate lemmas and proof strategies, perform mathematical checks, and
assist with drafting and revising the text.  The author independently
reconstructed and verified every argument, checked the cited sources, and made
all final mathematical and editorial decisions.  The author takes full
responsibility for the contents of the paper and for any errors.

\bigskip
\noindent Mathematics Department, University of Houston, Houston, TX, 77004, USA\\
E-mail: \texttt{clutsko@uh.edu}

\end{document}